\documentclass[10pt, a4paper]{amsart}

\usepackage[T1]{fontenc}
\usepackage[utf8]{inputenc}
\usepackage{amsmath, amssymb, amsfonts, amsthm, amsbsy}
\usepackage[margin=2.0cm]{geometry}
\usepackage[]{setspace}
\usepackage{enumerate}
\usepackage[,dvipsnames]{xcolor}
\usepackage[,unicode]{hyperref}
\hypersetup{
	urlcolor=blue,
	colorlinks=true,
	linkcolor= blue,
	citecolor=blue
}
\usepackage{comment, stmaryrd}
\usepackage[all]{xy}
\usepackage{libertine,libertinust1math}
\usepackage{tikz}
\usetikzlibrary{fit,shapes,calc,arrows,through,intersections,arrows.meta,decorations.pathmorphing}
\usepackage{tkz-euclide}

\theoremstyle{plain}
\newtheorem{theorem}{Theorem}[section]
\newtheorem{proposition}[theorem]{Proposition}
\newtheorem{lemma}[theorem]{Lemma}
\newtheorem{corollary}[theorem]{Corollary}
\newtheorem{claim}{Claim}[theorem]
\newtheorem{fact}[theorem]{Fact}

\theoremstyle{definition}
\newtheorem{example}[theorem]{Example}
\newtheorem{definition}[theorem]{Definition}
\newtheorem{question}[theorem]{Question}

\newtheorem{problem}[theorem]{Problem}
\theoremstyle{remark}
\newtheorem{remark}[theorem]{Remark}
\newtheorem{convention}[theorem]{Convention}

\newenvironment{claimproof}[1]{%
\par\noindent{\em Proof of claim.}\space#1}{\hfill $\blacksquare_{\text{claim}}$

}

\newcommand\rmLambda{\mathrm{\Lambda}}

\providecommand{\bfGamma}{\boldsymbol{\Gamma}}

\newcommand\impdeg{\mathfrak{imp}}
\newcommand\Impdeg{\mathfrak{Imp}}
\newcommand\fraki{\mathfrak{i}}
\newcommand\frakI{\mathfrak{I}}

\newcommand\IMP{\mathrm{IMP}} 
\newcommand\FIMP{\mathsf{IMP}}
\newcommand\VFIMP{\mathsf{AKE}\text{-}\mathsf{IMP}} 
\newcommand\Frag{{F}} 
\newcommand\FFrag{\mathsf{F}} 
\newcommand\Th{\mathrm{Th}}

\renewcommand\Form{\mathrm{Form}}
\newcommand\Sent{\mathrm{Sent}}
\newcommand\res{\mathrm{res}}
\newcommand\val{\mathrm{val}}

\newcommand\id{\mathrm{id}}
\newcommand\Lang{\mathfrak{L}}
\newcommand\Lring{\Lang_{\mathrm{ring}}}
\newcommand\Lval{\Lang_{\val}}

\newcommand\Loag{\Lang_{\mathrm{oag}}}
\newcommand\Lbasis{\Lang_{b}}
\newcommand\Lblambda{\Lang_{b,\lambda}}
\newcommand\LQ{\Lang_{Q}}
\newcommand\Llambda{\Lang_{\lambda}}
\newcommand\LlambdaP{\Lang_{p,\lambda}}
\newcommand\LlambdaPi{\Lang_{p,\fraki,\lambda}}
\newcommand\Lringlambda{\Lang_{\mathrm{ring},\lambda}}
\newcommand\LringlambdaP{\Lang_{\mathrm{ring},p,\lambda}}
\newcommand\LringlambdaPi{\Lang_{\mathrm{ring},p,\fraki,\lambda}}
\newcommand\Lvlambda{\Lang_{\val,\lambda}}

\newcommand\Lk{\Lang_{\mathbf{k}}}
\newcommand\LGamma{\Lang_{\bfGamma}}

\newcommand\ps[1]{{#1(\!(t)\!)}}
\newcommand\pI[1][k]{#1_{[p]}}
\newcommand\pB[1][k]{#1_{[\![p]\!]}}
\newcommand\ppower[1][k]{#1^{(p)}}
\newcommand\pspan[2][k]{\ppower[#1](#2)}

\newcommand{\AKE}{\mathrm{AKE}}
\newcommand{\sAKE}{\mathrm{s}\AKE}

\newcommand{\Fth}{\mathbf{F}}
\newcommand{\Xth}{\mathbf{X}}
\newcommand{\FL}{\mathbf{F}_{\lambda}}
\newcommand{\FpL}{\mathbf{F}_{p,\lambda}}

\newcommand{\ACF}{\mathbf{ACF}}
\newcommand{\SCF}{\mathbf{SCF}}

\newcommand{\SCVF}{\mathbf{SCVF}}

\newcommand{\STVF}{\mathbf{STVF}}
\newcommand{\eq}{\mathrm{eq}}

\newcommand\EP{$\rmLambda$\textrm{REP}}

\newcommand\ul[1]{\underline{#1}}
\newcommand\pII{{\rm I}}
\newcommand\pSS{{\rm S}}
\newcommand\pInd[1]{\pII_{#1}}
\newcommand\pl[1]{{l}_{#1}}
\renewcommand\pl[1]{{\rm L}_{#1}}
\newcommand\pL[1]{{\rm L}_{#1}}
\newcommand\bigL{\mathbf{L}}

\newcommand\bbF{\mathbb{F}}
\newcommand\bbA{\mathbb{A}}

\newcommand\Mod{\mathbf{Mod}}

\newcommand\perf{\mathrm{perf}}

\newcommand\rac{\mathrm{rac}}
\newcommand\acl{\mathrm{acl}}
\newcommand\dcl{\mathrm{dcl}}

\newcommand\Sep{\mathrm{Sep}}

\newcommand\Toplus{\oplus_{T}}

\newcommand\Teq{\simeq_{T}}

\newcommand\meq{\simeq_{m}}

\newcommand\pr{\mathrm{pr}}
\newcommand\locus{\mathrm{locus}}

\newcommand\concat[2]{#1^{\smallfrown}#2}

\newcommand\squaretall\blacklozenge
\renewcommand{\square}{\mathbin{\rotatebox[origin=c]{-90}{$\squaretall$}}}

\newcommand\localin{\lesssim}
\newcommand{\localeq}{\mathbin{\rotatebox[origin=c]{-180}{$\simeq$}}}

\newcommand\DAG{\hyperref[DAG]{{$(\dagger)$}}}

\title[On Lambda functions]{On Lambda functions in henselian and separably tame valued fields}
\author{Sylvy Anscombe}
\thanks{\today.}
\address{Universit\'{e} Paris Cité and Sorbonne Universit\'{e}, CNRS, IMJ-PRG, F-75013 Paris, France}
\email{sylvy.anscombe@imj-prg.fr}

\begin{document}
\begin{abstract}
	Given a field extension $F/C$, the ``Lambda closure'' $\rmLambda_{F}C$ of $C$ in $F$
is a subextension of $F/C$
that is minimal with respect to inclusion such that $F/\rmLambda_{F}C$ is separable.
The existence and uniqueness of $\rmLambda_{F}C$ was proved by Deveney and Mordeson in 1977.
We show that it admits a simple description in terms of given generators for $C$:
we expand the language of rings by the parameterized Lambda functions, and then $\rmLambda_{F}C$ is the subfield of $F$ generated over $C$ by additionally closing under these functions.
We then show that, given particular generators of $C$,
$\rmLambda_{F}C$ is the subfield of $F$ generated iteratively by the images of the generators under Lambda functions taken
with respect to $p$-independent tuples also drawn from those generators.

	We apply these results to given a ``local description'' of existentially definable sets in fields equipped with a henselian topology.
Let $X(K)$ be an existentially definable set in the theory of a field $K$ equipped with a henselian topology $\tau$.
We show that there is a definable injection into $X(K)$ from a Zariski-open subset $U_{1}^{\circ}$ of a set with nonempty $\tau$-interior,
and that each element of $U_{1}^{\circ}$ is interalgebraic (over parameters) with its image in $X(K)$.
This can be seen as a kind of {\em very weak local quantifier elimination}, and it shows that existentially definable sets are (at least generically and locally) definably pararameterized by ``big'' sets.

	In the final section
we extend the theory of Separably Tame valued fields,
developed by Kuhlmann and Pal,
to include the case of infinite degree of imperfection, and to allow expansions of the residue field and value group structures.
We prove an embedding theorem which allows us to deduce the usual kinds of resplendent Ax--Kochen/Ershov principles.
\end{abstract}
\maketitle
\setcounter{secnumdepth}{2}
\setcounter{tocdepth}{3}

\section{Introduction}

We study the ``Lambda closure''
$\rmLambda_{F}C$
of a given field extension $F/C$:
this is the smallest subfield of $F$, containing $C$, such that $F/\rmLambda_{F}C$ is separable.
Though its existence and uniqueness were established in a 1977 paper of
Deveney and Mordeson
(\cite{DM}),
it remains---in the view of this author---less well-known than it should be.
In principle, $\rmLambda_{F}C$ is obtained by recursively closing $C$ under the so-called
``parameterized Lambda functions''
(see Definition~\ref{def:lambda}),
relative to every subset of $C$ that is $p$-independent in $F$,
using the terminology of $p$-independence developed by Mac Lane, and others, nearly a century ago.
We reduce the complexity of this process by showing that it suffices to iteratively adjoin the image of one set of generators under the Lambda functions relative to (the finite subsets of) one maximal $p$-independent subset.
This modest efficiency yields a clean description of $\rmLambda_{F}C$ in terms of a given generating set of $C$.
For example, if $F/C$ is already separable and $c$ is a well-ordered $p$-basis of $C$,
we develop from any well-ordered subset $a$ of $F$
a well-ordered set $\lambda_{F/c}a$
(see Definition~\ref{def:lambda_main}),
called the {\em local Lambda closure} of $a$,
such that
$C(\lambda_{F/c}a)=\rmLambda_{F}C(a)$,
and yet nevertheless
each element of
$\lambda_{F/c}a$
is existentially definable in $F$ over $c\cup a$ in the first-order language $\Lring$ of rings.
The local Lambda closure appears in the author's PhD thesis,
but also
in Hong's important paper~\cite{Hong}
(where it is called the ``relative $\lambda$-resolvant'')
in which a quantifier-elimination for separably closed valued fields is proved.

Pursuing this point of view,
we employ a language
$\Llambda$
consisting of function symbols
$\lambda_{m}(x,y)$,
for $m\in\mathbb{N}$.
This language is not new:
it and its variants have been used before to study separably closed fields and separably closed valued fields
(see remarks~\ref{rem:SCF} and~\ref{rem:SCVF}),
although we give a presentation
in Definition~\ref{def:language}
that is independent of characteristic.
Any field admits a natural $\Llambda$-structure
via interpreting the new symbols
$\lambda_{m}(x,y)$
by the parameterized Lambda functions.
It is clear that $\rmLambda_{F}C$ is the
$(\Lring\cup\Llambda)$-substructure
of $F$ generated by $C$, 
but we show that it is the subfield generated by the image of $C$ under $\Llambda$-terms,
thus separating the function symbols of $\Lring$ from those of $\Llambda$,
and so expressing each $\Lring\cup\Llambda$-term as the composition of an $\Lring$-term with an $\Llambda$-term.
The following theorem,
proved in section~\ref{section:Lambda},
is a summary:

\begin{theorem}\label{thm:intro_1}
Let $F/C$ be separable, let $c$ be a well-ordered $p$-basis of $C$, and let $a$ be a well-ordered subset of $F$.
There exists a
subset
$\lambda_{F/c}a$
of $F$
such that
\begin{enumerate}[{\bf(i)}]
\item
$\rmLambda_{F}C(a)=C(\lambda_{F/c}a)$ and
\item
$\lambda_{F/c}a$ is a collection of closed $\Llambda$-terms in the elements of $c\cup a$.
\end{enumerate}
If moreover $a,b\subseteq F$ are finite and $F=C(a,b)$, then
\begin{enumerate}[{\bf(i)}]
\setcounter{enumi}{2}
\item
there is a finite subset $\lambda_{F/b/c}a\subseteq\lambda_{F/c}a$ such that $\rmLambda_{F}C(a)=C(\lambda_{F/b/c}a)$.
\end{enumerate}
\end{theorem}

As a first application, we use
the local Lambda closure
to describe Diophantine subsets of a field
$K$
equipped with a henselian topology
$\tau$,
as defined in \cite{PZ78}:
these are the topologies that are ``locally equivalent'', as defined in that paper (see \cite[\S1]{PZ78}), to a topology induced by a nontrivial henselian valuation.
This includes the case that $\tau$ is itself induced by a nontrivial henselian valuation on $K$.
Henselianity, at this topological level, is equivalent to
the Implicit Function Theorem {\em for polynomials},
as elucidated in \cite[(7.4) Theorem]{PZ78}.
Adapting the terminology of
\cite[7, Chapter II, Section 3]{Lang_geometry},
the {\em locus} of a tuple%
\footnote{In this paper, a {\em tuple} is a finite totally ordered set.
Usually, but not always, tuples are elements of finite Cartesian products of a given set.}
$a\in K^{m}$ over a subfield $C\subseteq K$, denoted $\locus(a/C)$,
is the smallest Zariski closed subvariety of $\bbA_{C}^{m}$, defined over $C$, of which $a$ is a rational point,
see Definition~\ref{def:locus}.
If $b\in K^{n}$ is another tuple
then, just as in Theorem~\ref{thm:intro_1}~{\bf(iii)},
there is a finite subset
$\lambda_{K/b/c}a$ of $\lambda_{K/c}a$
such that
$C(\lambda_{K/b/c}a,b)/C(\lambda_{K/b/c}a)$
is separable.
For $\ell=|\lambda_{K/b/c}a|$,
there is a
coordinate projection
$\sigma_{K/b/c}:\bbA^{\ell}_{C}\rightarrow\bbA^{m}_{C}$
that
maps
$\lambda_{K/b/c}a\mapsto a$.
Using these tools,
in section~\ref{section:IFT} we prove the following theorem
and its corollary.

\newcommand\theoremintrothree{%
Let $K/C$ be a separable field extension,
let $\tau$ be a henselian topology on $K$,
let $c$ be a well-ordered $p$-basis of $C$,
let $X\subseteq K^{m}$ be an existentially $\Lring(C)$-definable set,
and let $a\in X$.
There exists a $\tau$-neighbourhood $U$ of
$\lambda_{K/b/c}a$
such that
$X$ contains the image of
$$\locus(\lambda_{K/b/c}a/C)(K)\cap U$$
under the coordinate projection $\sigma_{K/b/c}$.
}

\begin{theorem}\label{thm:intro_3}
\theoremintrothree
\end{theorem}

This theorem simplifies problems around existential definability in fields equipped with henselian topologies,
extending the methods of \cite{A19}.

\begin{corollary}\label{cor:intro_b}
Let $K$ be a field with a henselian topology $\tau$,
and let $B\subseteq K$ be a subset.
Then the model-theoretic existential $\Lring$-algebraic closure of $B$ in $K$
is a subset of the
the relative algebraic closure of $\rmLambda_{K}\bbF(B)$ in $K$,
where $\bbF$ is the prime subfield.
\end{corollary}

The theory
$\STVF$
of
separably tame valued fields
was introduced,
along with that of tame valued fields,
by
Kuhlmann
in
\cite{Kuh16},
and further developed by Kuhlmann and Pal
in \cite{KuhlmannPal}.
Separably tame---and especially tame---valued fields
feature widely in contemporary research on the model theory of valued fields.
For example they are crucial to the arguments and results in
\cite{AF16,AJ_Henselianity,AJ_Cohen,AJ_NIP,AK,BK,Daans,Jahnke_expansions,JK23,JS,Kartas_tame,Kartas_tilting,KR,Sinclair,vdS},
to name just a few.
In \cite{KuhlmannPal},
using an expansion called $\LQ$ of a standard language of valued fields $\Lval$,
Kuhlmann and Pal prove that a range of ``Ax--Kochen/Ershov principles'' hold
--- generalizing the original principles for henselian valued fields of equal characteristic zero
(\cite{AxKochen-I,Ershov}):

\newcommand\theoremintroKP{%
The class $\Mod(\STVF_{p,\fraki})$ of all separably tame valued fields of fixed characteristic $p>0$ and fixed finite imperfection degree $\fraki\in\mathbb{N}$
is an $\AKE^{\exists}$-class
in $\LQ$,
an $\AKE^{\prec}$-class
in $\LQ$,
and an $\AKE^{\equiv}$-class
in $\Lval$.}

\begin{theorem}[{\cite[Theorem 1.2]{KuhlmannPal}}]\label{thm:intro_KP}
\theoremintroKP
\end{theorem}
\noindent
For precise definitions of these theories and properties, see section~\ref{section:STVF}.
We reformulate and prove a version
(Theorem~\ref{thm:EP})
of the usual Embedding Lemma for separably tame valued fields,
in order that it apply to all separably tame valued fields of equal characteristic,
regardless of imperfection degree, and for it to yield separability of the resultant embedding.
Combining $\Lval$ 
with the language $\Llambda$,
we obtain
$\Lvlambda$.
In this language, and using our new Embedding Lemma,
we prove
Theorem~\ref{thm:STVF_main},
which gives various transfer statements for fragments of theories,
between two models,
over a common defectless substructure.
We then deduce Theorem~\ref{thm:sAKE}
which gives a range of
``separable Ax--Kochen/Ershov principles''
(sAKE, see Definition~\ref{def:sAKE}).
Throughout,
we work with the $\Lvlambda$-theory
$\STVF^{\eq}$
of equal characteristic separably tame valued fields,
with emphasis on equal positive characteristic,
since we say nothing new in
mixed characteristic
or in equal characteristic zero.
Nevertheless, our results are uniform in the (equal) characteristic.

Our main theorems in this direction
are Theorem~\ref{thm:STVF_main} and~\ref{thm:sAKE}.
Of the latter the following is a special case:

\begin{theorem}\label{thm:intro_2}
Let $\square\in\{\equiv,\equiv_{\exists},\preceq,\preceq_{\exists}\}$.
The class of all
separably tame valued fields of equal characteristic
is an
$\sAKE^{\square}$-class
for
the triple of languages
$(\Lvlambda,\Lring,\Loag)$,
that is:

Let $(K,v),(L,w)\models\STVF^{\eq}$,
and additionally suppose $(K,v)\subseteq(L,w)$ in case $\square$ is either $\preceq$ or $\preceq_{\exists}$.
Then
\begin{itemize}
\item
$(K,v)\square(L,w)$
in $\Lvlambda$
\end{itemize}
if and only if
\begin{itemize}
\item
$Kv\square Lw$
in $\Lring$,
\item
$vK\square wL$
in the language $\Loag$ of ordered abelian groups,
and
\item
$K$ and $L$ have the same elementary imperfection degree
(as defined in~\ref{section:lambda_language}).
\end{itemize}
\end{theorem}

\noindent
Arguably the result finds its most familiar form by taking $\square$ to be $\equiv$:
{\em
	For all $(K,v),(L,w)\models\STVF^{\eq}$
	we have
	$(K,v)\equiv(L,w)$
	in $\Lvlambda$
	if and only if
	$Kv\equiv Lw$ in $\Lring$,
	$vK\equiv wL$ in $\Loag$,
	and
	$K$ and $L$ have the same elementary imperfection degree.
}

\noindent
We prove this theorem, and the following corollary, in section~\ref{section:STVF}.

\begin{corollary}\label{cor:intro_a}
Let $(K,v)$ be a separably tame valued field of equal characteristic.
Then
\begin{itemize}
\item
the theory of $(K,v)$ in the language $\Lval$ of valued fields
is decidable
\end{itemize}
if and only if 
\begin{itemize}
\item
the theory of $Kv$ in the language $\Lring$ of rings is decidable and
\item
the theory of $vK$ in the language $\Loag$ of ordered abelian groups is decidable.
\end{itemize}
\end{corollary}

These results on separably tame valued fields,
especially
Theorems~\ref{thm:STVF_main} and~\ref{thm:sAKE},
extend
those of
\cite{KuhlmannPal}
in two ways:
firstly, we allow infinite imperfection degree,
and secondly, our results are resplendent,
though resplendency in finite imperfection degree can be read from the arguments presented in \cite{KuhlmannPal}.

\section{Lambda closure}
\label{section:Lambda}

The topic of separability and inseparability
pertains mainly to the world of fields of positive characteristic,
since all field extensions in characteristic zero are separable,
according to the definition that we recall below.
Nevertheless, it is possible for us to give a treatment that includes the case of characteristic zero,
by making the convention that throughout this section
$p$ will denote the {\em characteristic exponent}
of the integral domain in question, 
i.e.~$p$ is the characteristic if this is positive, and $p$ is $1$ otherwise.
Throughout $\bbF$ denotes the prime field of characteristic exponent $p$,
i.e.~$\bbF_{p}$ for $p>1$, and $\mathbb{Q}$ for $p=1$.
For $n<\omega$, and a subset $A$ of a field $F$, we write $A^{(n)}:=\{a^{n}\mid a\in A\}$,
and for a choice $F^{\mathrm{alg}}$ of algebraic closure of $F$ we write
$A^{(p^{-n})}:=\{a\in F^{\mathrm{alg}}\mid a^{p^{n}}\in A\}$.
The {\em perfect hull} of $A$, denoted $A^{\perf}$, is the directed union of the sets $A^{(p^{-n})}$, for $n<\omega$.
Thus in characteristic zero, we always have $A^{(p^{-n})}=A$ and $A^{\perf}=A$.

A field extension
$F/C$
of characteristic exponent $p$
is {\em separably generated}
if there is a transcendence basis $a\subseteq F$ of $F/C$ such that $F$ is separably algebraic over $C(a)$,
such a basis is called a {\em separating transcendence basis} of the extension.
We say that $F/C$ is {\em separable} if $F/C$ is linearly disjoint from $C^{(p^{-1})}/C$,
and we say that $F/C$ is {\em separated} if additionally $F=\ppower[F]C$.
An embedding $\varphi:C\rightarrow F$ is {\em separable} if $F/\varphi(C)$ is separable.
A field $F$ is {\em perfect} if every extension of it is separable, and this holds if and only if $F=\ppower[F]$,
which in turn is equivalent to $F=F^{\perf}$.

The following theorem,
originally due to Mac Lane,
is the first characterization of separable field extensions.

\begin{lemma}[{see \cite[Chapter VIII, Proposition 4.1]{Lang}}]\label{lem:Mac_Lane_I}
For a field extension $F/C$,
the following are equivalent.
\begin{enumerate}[{\bf(i)}]
\item
$F/C$ is separable.
\item
$F/C$ and $C^{\perf}/C$ are linearly disjoint.
\item
Every finitely generated subextension $E/C$ of $F/C$ is separably generated.
\item
Every finite subset $A$ of $E$ may be refined to a separating transcendence basis of $C(A)/C$.
\end{enumerate}
\end{lemma}

A subset $A$ of $F$ is 
{\em $p$-independent over $C$}
if
$a\notin\ppower[F]C(A\setminus\{a\})$,
for all $a\in A$.
It is
{\em $p$-spanning over $C$}
if
$F=\ppower[F]C(A)$,
and it is a
{\em $p$-basis over $C$}
if it is both $p$-independent and $p$-spanning over $C$.
Mostly we will be interested in the absolute versions of these notions,
i.e.~when $C$ is a prime field $\bbF$,
in which case we simply say {\em $p$-independent}, {\em $p$-spanning}, and {\em $p$-basis}.
We see immediately that $p$-independence is of finite character:
$A\subseteq F$ is $p$-independent in $F$ over $C$ if and only if every finite subset of $A$ is $p$-independent in $F$ over $C$,
since $\ppower[F]C(A)$ is the union of subfields $\ppower[F]C(a)$ for finite subsets $a\subseteq A$.
Typically, though not always, we will work with well-ordered subsets of fields, rather than (unordered) subsets,
though this is not of any real significance until we define the ``splitting pairs'' map in~\ref{section:local_lambda_closure}.
We write $\pI[(F/C)]$ (respectively $\pB[(F/C)]$) for the set of well-ordered subsets of $F$
that are $p$-independent (resp.~$p$-bases) in $F$ over $C$,
and we write $\pI[F]$ (resp.~$\pB[F]$) in the absolute case when $C=\bbF$.
The relation of $p$-independence in $F$ over $C$ satisfies the exchange property:
that is $a\in\ppower[F]C(A,b)\setminus\ppower[F]C(A)$ implies $b\in\ppower[F]C(A,a)$.
It defines a pre-geometry on subsets of $F$,
and any two $p$-bases (in $F$ over $C$) have the same cardinality:
thus we may define the {\em (relative) imperfection degree}
of $F$ over $C$,
denoted $\impdeg(F/C)$,
to be the cardinality of a $p$-basis of $F$ over $C$.
The {\em imperfection degree}
of $F$,
denoted $\impdeg(F):=\impdeg(F/\bbF)$,
is the cardinality of a $p$-basis of $F$ in the absolute case.
If $\impdeg(F/C)$ is finite then $[F:\ppower[F]C]=p^{\impdeg(F/C)}$,
and $[F:\ppower[F]C]=\impdeg(F/C)$ otherwise.

The following lemma gives a second characterization of separable field extensions,
this time in terms of $p$-independence,
and is also due to Mac Lane.

\begin{lemma}[{cf \cite[Theorems 7 and 10]{ML39}}]\label{lem:Mac_Lane_II}
For a field extension $F/C$,
the following are equivalent.
\begin{enumerate}[{\bf(i)}]
\item
$F/C$ is separable.
\item
Every $p$-independent subset of $C$ is $p$-independent in $F$,
equivalently
$\pI[C]\subseteq\pI[F]$.
\item
Every finite $p$-independent subset of $C$ is $p$-independent in $F$.
\item
Every $p$-basis of $C$ is $p$-independent in $F$,
equivalently
$\pB[C]\subseteq\pI[F]$.
\item
Some $p$-basis of $C$ is $p$-independent in $F$,
equivalently
$\pB[C]\cap\pI[F]\neq\emptyset$.
\end{enumerate}
\end{lemma}

The implication {\bf(iii)}$\Rightarrow${\bf(ii)} follows from the finite character of $p$-independence.
Trivially, the previous lemma has the following analogue for separated extensions.

\begin{lemma}\label{lem:Mac_Lane_IIa}
For a field extension $F/C$,
the following are equivalent.
\begin{enumerate}[{\bf(i)}]
\item
$F/C$ is separated.
\item
Every $p$-basis of $C$ is a $p$-basis of $F$,
equivalently
$\pB[C]\subseteq\pB[F]$.
\item
Some $p$-basis of $C$ is a $p$-basis of $F$,
equivalently
$\pB[C]\cap\pB[F]\neq\emptyset$.
\end{enumerate}
\end{lemma}

Given a cardinal $\fraki$,
the set of
{\em finitely supported multi-indices}
(indexed by $\fraki$,
with each index $<p$)
is 
\begin{align*}
	p^{[\fraki]}
&=
	\{I=(i_{\alpha})_{\alpha<\fraki}\in\{0,\ldots,p-1\}^{\fraki}\mid\textrm{$\mathrm{supp}(I)$ is finite}\},
\end{align*}
where $\mathrm{supp}(I):=\{\alpha<\fraki\mid i_{\alpha}\neq0\}$.
Given a subset $b=(b_{\alpha})_{\alpha<\fraki}$ of a ring,
indexed by $\fraki$,
a {\em $p$-monomial} in $b$
is the product
$b^{I}:=\prod_{\alpha<\fraki}b_{\alpha}^{i_{\alpha}}$,
for some $I=(i_{\alpha})_{\alpha<\fraki}\in p^{[\fraki]}$.
We observe that every well-ordered set $b$ has a unique order-preserving indexing by $|b|$.

\begin{lemma}\label{lem:p_linearity}
For a well-ordered subset $b\subseteq F$,
we have the following:
\begin{enumerate}[{\bf(i)}]
\item
$b\in\pI[(F/C)]$
if and only if
$\{b^{I}\mid I\in p^{[|b|]}\}$
is an $\ppower[F]C$-linear base for $\ppower[F]C(b)$, and
\item
$b\in\pB[(F/C)]$
if and only if
$\{b^{I}\mid I\in p^{[|b|]}\}$
is an $\ppower[F]C$-linear base for $F$.
\end{enumerate}
In particular,
if $b\in\pB[F]$,
then
$F=\bigoplus_{I\in p^{[|b|]}}b^{I}F^{(p)}$
is a direct sum of $F^{(p)}$-vector spaces.
\end{lemma}
\begin{proof}
If $b$ is $p$-independent in $F$ over $C$,
then each $b_{\alpha}$ generates a purely inseparable field extension of degree $p$
over $\pspan[F]C(b\setminus\{b_{\alpha}\})$.
The rest follows from the usual Tower Lemma that describes linear bases of iterated field extensions.
\end{proof}

This lemma enables the following definition, which is central to everything that follows.

\begin{definition}[{Lambda functions}]\label{def:lambda}
For $b\in\pI[F]$
and for $a\in\pspan[F]{b}$,
there is a unique family
$(\lambda_{I}^{b}(a))_{I\in p^{[|b|]}}$
of elements of $F$
such that
\begin{align*}
	a
&=
	\sum_{I\in p^{[|b|]}}b^{I}\lambda_{I}^{b}(a)^{p}.
\end{align*}
Thus for each $I\in p^{[|b|]}$ there is a function
\begin{align*}
	\lambda_{I}^{b}:\ppower[F](b)&\rightarrow F\\
	a&\mapsto\lambda_{I}^{b}(a).
\end{align*}
We write $\lambda^{b}$ for the function
$a\mapsto(\lambda_{I}^{b}(a))_{I\in p^{[|b|]}}$
from $\pspan[F]{b}$ to the set of subsets of $F$ indexed by $p^{[|b|]}$.
On the other hand, the {\em parameterized lambda functions}
are the partial functions
$\lambda_{I}:F\times\pI[F]\rightarrow F$,
for $I\in p^{[\fraki]}$,
that are defined by 
$\lambda_{I}(a,b):=\lambda_{I}^{b}(a)$
when $|b|=\fraki$
and $a\in\ppower[F](b)$,
and are undefined otherwise.
Finally, for any set $A\subseteq F$,
we will write $\lambda^{b}(A)$ to mean the union 
$\bigcup_{I\in p^{[|b|]}}\lambda^{b}_{I}(A\cap\pspan[F]{b})$,
where each $\lambda^{b}_{I}(A\cap\pspan[F]{b})$ is simply the set
$\{\lambda^{b}_{I}(a)\mid a\in A\cap\pspan[F]{b}\}$.
\end{definition}

\begin{remark}\label{rem:finite_character}
The finite character of $p$-independence appears in a second guise:
the set $\lambda^{b}(A)$ is the union of sets $\lambda^{b_{0}}(A)$
for finite subsets $b_{0}\subseteq b$.
\end{remark}

The next proposition gives a third and final characterization of separable field extensions.
It is certainly well known, dating back to at least the work of Mac Lane,
however we give a proof for the convenience of the reader.

\begin{proposition}\label{prp:Mac_Lane_III}
For a field extension $F/C$,
the following are equivalent.
\begin{enumerate}[{\bf(i)}]
\item
$F/C$ is separable.
\item
$\pspan[F]{b}\cap C=\ppower[C](b)$
for each
$b\in\pI[F]$
with
$b\subseteq C$.
\item
$\pspan[F]{b}\cap C=\ppower[C](b)$
for each
finite
$b\in\pI[F]$
with
$b\subseteq C$.
\item
$\lambda^{b}(\pspan[F]{b}\cap C)\subseteq C$
for each
$b\in\pI[F]$
with
$b\subseteq C$.
\item
$\lambda^{b}(\pspan[F]{b}\cap C)\subseteq C$
for each
finite
$b\in\pI[F]$
with
$b\subseteq C$.
\end{enumerate}
\end{proposition}
\begin{proof}
For $n\in\mathbb{N}$,
we denote
\begin{enumerate}
\item[{\bf(i)$_{n}$}]
for each $n$-tuple $b$ from $C$,
if $b\in\pI[C]$ then $b\in\pI[F]$,
\end{enumerate}
and
\begin{enumerate}
\item[{\bf(iii)$_{n}$}]
$\pspan[F]{b}\cap C=\ppower[C](b)$
for each $n$-tuple $b\in\pI[F]$ with $b\subseteq C$.
\end{enumerate}
Note that
{\bf(i)$_{0}$}
is true unconditionally,
and
{\bf(i)} is equivalent to 
$\bigwedge_{n\in\mathbb{N}}${\bf(i)$_{n}$}
by Lemma~\ref{lem:Mac_Lane_II}~{\bf(i)}$\Leftrightarrow${\bf(iii)}.
The equivalences
{\bf(ii)}$\Leftrightarrow${\bf(iii)}
and
{\bf(iv)}$\Leftrightarrow${\bf(v)}
also follow from the finite character of $p$-independence
(see Remark~\ref{rem:finite_character}).
Clearly
{\bf(iii)} is equivalent to 
$\bigwedge_{n\in\mathbb{N}}${\bf(iii)$_{n}$}.
We will show
{\bf(i)$_{n+1}$}$\Rightarrow${\bf(iii)$_{n}$}.
Let $b\in\pI[F]$ be an $n$-tuple with $b\subseteq C$.
Clearly
$\pspan[F]{b}\cap C\supseteq\ppower[C](b)$.
For each
$c\in(\pspan[F]{b}\cap C)\setminus\ppower[C](b)$,
we have
$\concat{b}{c}\in\pI[C]\setminus\pI[F]$
(where $\concat{b}{c}$ is the concatenation of $b$ and $c$),
which contracts 
{\bf(i)$_{n+1}$}.
Thus 
$\pspan[F]{b}\cap C=\ppower[C](b)$,
i.e.~{\bf(iii)$_{n}$} holds.
The implication
{\bf(i)}$\Rightarrow${\bf(iii)}
follows.
Conversely,
we suppose as an inductive hypothesis
that 
$\bigwedge_{i<n}${\bf(iii)$_{i}$}$\Rightarrow\bigwedge_{i<n+1}${\bf(i)$_{i+1}$}
--- note that the base case $n=0$ is trivial.
We will show
$\bigwedge_{i\leq n}${\bf(iii)$_{i}$}$\Rightarrow\bigwedge_{i\leq n+1}${\bf(i)$_{i+1}$},
so we suppose
{\bf(iii)$_{i}$}
for all $i\leq n$.
Let $b\subseteq C$ be an $n$-tuple,
let $c\in C$,
and suppose that $\concat{b}{c}\in\pI[C]$.
Then in particular
$b\in\pI[C]$,
so $b\in\pI[F]$
by
{\bf(i)$_{n}$}.
Then
by
{\bf(iii)$_{n}$}
we have
$\pspan[F]{b}\cap C=\ppower[C](b)$.
Thus $c\notin\pspan[F]{b}$, so $\concat{b}{c}\in\pI[F]$,
which proves
{\bf(i)$_{n+1}$}.
Together this has proved
{\bf(iii)}$\Rightarrow${\bf(i)}.
To see
{\bf(iii)}$\Leftrightarrow${\bf(v)},
we let $b$ be an $n$-tuple from $C$
with $b\in\pI[F]$.
As before, it is clear that
$\pspan[F]{b}\cap C\supseteq\ppower[C](b)$.
Then
$\pspan[F]{b}\cap C\subseteq\ppower[C](b)$
if and only if
$\lambda^{b}_{I}(a)\in C$
for all $I\in p^{[|b|]}$ and all $a\in\ppower[F](b)\cap C$,
by Lemma~\ref{lem:p_linearity}.
This shows
{\bf(iii)}$\Leftrightarrow${\bf(v)}.
\end{proof}

Observe that
Proposition~\ref{prp:Mac_Lane_III}~{\bf(iv)}
expresses that $C$ is ``$\rmLambda$-closed'' in $F$,
i.e.~closed under all the lambda functions with respect to well-ordered subsets of $C$ that are $p$-independent in $F$.
By~{\bf(v)}, it is equivalent that $C$ be closed under those lambda functions with respect to finite subsets of $C$ that are $p$-independent in $F$.

We are familiar with the elementary fact that every algebraic field extension
$F/C$
decomposes uniquely into a separably algebraic extension $E/C$ and a purely inseparable extension $F/E$.
The reverse is not true: if $F/C$ is not normal then there is not necessarily subextension $E/C$ that is purely inseparable, such that $F/E$ is separable.
However, as the following theorem shows, it is still true, even for arbitrary field extensions $F/C$,
that there is a minimal subextension $E/C$ such that $F/E$ is separable.

\begin{theorem}[{cf \cite[Theorem 1.1]{DM}}]\label{thm:DM}
For every field extension $F/C$ there is a miminum element
of
$$\Sep(F/C):=\{D\mid\text{$D\subseteq F$ is a subfield, with $F/D$ separable and $C\subseteq F$}\}$$
with respect to inclusion.
\end{theorem}

\begin{definition}\label{def:Lambda_closure}
For any field extension $F/C$,
we denote by
$\rmLambda_{F}C$
the minimum 
element of
$\Sep(F/C)$
and we call it the {\em $\rmLambda$-closure of $C$ in $F$}.
\end{definition}

\begin{remark}
We note that $\rmLambda_{F}$ is a closure operation on the set of subfields of $F$,
since
$C\subseteq\rmLambda_{F}C$,
$\rmLambda_{F}C=\rmLambda_{F}\rmLambda_{F}C$,
and $C_{1}\subseteq C_{2}\implies\rmLambda_{F}C_{1}\subseteq\rmLambda_{F}C_{2}$.
\end{remark}

For a field extension $F/C$,
we denote by
$\rmLambda_{F}^{1}C$
the subfield of $F$ generated over $C$
by the elements
$\lambda_{I}^{b}(a)$,
for every
finite
$b\in\pI[F]$ with $b\subseteq C$, $a\in\pspan[F]{b}\cap C$, and $I\in p^{[|b|]}$.
By writing $\rmLambda_{F}^{0}C:=C$
and $\rmLambda_{F}^{n+1}C:=\rmLambda_{F}^{1}\rmLambda_{F}^{n}C$
we have recursively constructed an increasing chain of subfields of $F$ containing $C$.

\begin{example}
For subfields $C$ of a perfect field $F$, of characteristic exponent $p$,
$\rmLambda_{F}^{1}$
simply amounts to adjoining $p$-th roots:
$\rmLambda_{F}^{1}C=C^{(p^{-1})}$.
In particular, 
if $C$ is also perfect and $t\in F$,
then
$\rmLambda_{F}^{1}C(t)=C(t^{p^{-1}})$.
\end{example}

\begin{lemma}\label{lem:big_lambda}
$\rmLambda_{F}C$ is the directed union $\bigcup_{n<\omega}\rmLambda_{F}^{n}C$.
\end{lemma}
\begin{proof}
For convenience, let us denote
$L:=\bigcup_{n<\omega}\rmLambda^{n}_{F}C$.
By the definition of $\rmLambda^{1}_{F}C$ and the characterization of separability given in Proposition~\ref{prp:Mac_Lane_III}~{\bf(v)},
it is clear that $\rmLambda_{F}^{1}C\subseteq\rmLambda_{F}C$.
A simple induction yields
$\rmLambda_{F}^{n}C\subseteq\rmLambda_{F}C$
for each $n<\omega$,
and so $L\subseteq\rmLambda_{F}C$.
It remains to show that $F/L$ is separable,
to which end we will verify the criterion of
Proposition~\ref{prp:Mac_Lane_III}~{\bf(v)}.
Let
$b\in\pI[F]$
be finite, with
$b\subseteq L$,
and let $a\in\ppower[F](b)\cap L$.
Let $n<\omega$ be such that
$b\subseteq\rmLambda^{n}_{F}C$
and
$a\in\rmLambda^{n}_{F}C$.
Then
$$\lambda^{b}_{I}(a)\in\rmLambda^{1}_{F}\rmLambda^{n}_{F}C=\rmLambda^{n+1}_{F}C\subseteq L,$$
for each $I\in p^{[|b|]}$.
This verifies
Proposition~\ref{prp:Mac_Lane_III}~{\bf(v)}
for the extension $F/L$,
which shows that $F/L$ is separable,
and thus $\rmLambda_{F}C\subseteq L$, whence $\rmLambda_{F}C=L$.
\end{proof}

\begin{example}
For subfields $C$ of a perfect field $F$, of characteristic exponent $p$,
$\rmLambda_{F}$
simply amounts to taking the perfect hull:
$\rmLambda_{F}C=C^{\perf}$.
In this case,
$C=\rmLambda_{F}C$ if and only if $C$ is perfect.
\end{example}

\begin{remark}\label{rem:lambda_1}
Let $F$ be any field.
\begin{enumerate}[{\bf(i)}]
\item
If $(C_{i})_{i\in I}$ is a directed system of subfields of $F$,
then $\rmLambda_{F}\bigcup_{i\in I}C_{i}=\bigcup_{i\in I}\rmLambda_{F}C_{i}$,
which implies that this closure operation is finitary.
\item
Let $C\subseteq E\subseteq F$ be a tower of subfields of $F$
with $F/E$ separable.
Then
$\rmLambda_{K}C=\rmLambda_{F}C$.
For example, if $F^{*}\succeq F$ is an elementary extension,
then 
$\rmLambda_{F^{*}}C=\rmLambda_{F}C$.
\end{enumerate}
\end{remark}

\subsection{Some lambda algebra}

For the rest of this section we suppose that
$F/C$
is a separable field extension
of characteristic exponent $p$.
For a subset $A\subseteq F$
and a subring $R\subseteq F$, 
we denote by $R[A]$ the subring of $F$ generated by $R\cup A$.
Similarly, for a subfield $E\subseteq F$, $E(A)$ denotes the subfield of $F$ generated by $E\cup A$.
Perhaps it is helpful to reinforce that in the following lemma, we write $a_{1}a_{2}$ for the usual multiplicative product of $a_{1}$ and $a_{2}$ in the field $F$.

\begin{lemma}\label{lem:p-algebra}
Let $b\in\pI[F]$
and let $a,a_{1},a_{2}\in\pspan[F]{b}$.
Then
\begin{enumerate}[{\bf(i)}]
\item
$a\in\bbF[\lambda^{b}(a)]^{(p)}[b]$,
\item
$\lambda^{b}$ is an indexed family of additive homomorphisms,
and
\item
$\lambda^{b}(a_{1}a_{2})\subseteq\bbF[b,\lambda^{b}(a_{1}),\lambda^{b}(a_{2})]$.
\end{enumerate}
\end{lemma}
\begin{proof}
The first claim (which is a kind of ``warm up'') follows trivially from Definition~\ref{def:lambda}.
For the second claim:
for each $I\in p^{[|b|]}$, $\lambda^{b}_{I}$ is the composition of a coordinate function with the inverse of the Frobenius map, which are both additive homomorphisms.
For the final claim, we notice the following:
\begin{align*}
	a_{1}a_{2}
&=
	\big(\sum_{I_{1}\in p^{[|b|]}}b^{I_{1}}\lambda^{b}_{I_{1}}(a_{1})^{p}\big)\big(\sum_{I_{2}\in p^{[|b|]}}b^{I_{2}}\lambda^{b}_{I_{2}}(a_{2})^{p}\big)
\\
&=
	\sum_{I_{1},I_{2}\in p^{[|b|]}}b^{I_{1}+I_{2}}\lambda^{b}_{I_{1}}(a_{1})^{p}\lambda^{b}_{I_{2}}(a_{2})^{p}
\\
&=
	\sum_{J_{2}\in p^{[|b|]}}b^{J_{2}}\big(\sum_{J_{1}}b^{J_{1}}\lambda^{b}_{I_{1}}(a_{1})\lambda^{b}_{I_{2}}(a_{2})\big)^{p},
\end{align*}
where the second sum in the last line ranges over finitely supported multi-indices $J_{1}$, indexed by $|b|$ and with each index a natural number, such that
$I_{1}+I_{2}=J_{1}p+J_{2}$ and $J_{2}\in p^{[|b|]}$,
using coordinatewise addition of multi-indices.
Thus $\lambda^{b}_{I}(a_{1}a_{2})=\sum_{J_{1}}b^{J_{1}}\lambda^{b}_{I_{1}}(a_{1})\lambda^{b}_{I_{2}}(a_{2})$.
\end{proof}

\begin{lemma}\label{lem:lambda_calculus}
Let $b\in\pI[F]$ and let $a\subseteq\pspan[F]{b}$.
Then
\begin{enumerate}[{\bf(i)}]
\item
$\lambda^{b}(\bbF[a])\subseteq\bbF[b,\lambda^{b}(a)]$ and
\item
$\lambda^{b}(\bbF(a))\subseteq\bbF(b,\lambda^{b}(a))$.
\end{enumerate}
\end{lemma}
\begin{proof}
{\bf(i)} follows straight from Lemma~\ref{lem:p-algebra}~{\bf(ii,iii)}.
For {\bf(ii)},
by {\bf(i)} we have
$\bbF[a]\subseteq R^{(p)}[b]$,
where $R=\bbF[b,\lambda^{b}(a)]$.
By passing to the fields of fractions we have
$\bbF(a)\subseteq\mathrm{Frac}(R)^{(p)}(b)=\bbF(b,\lambda^{b}(a))^{(p)}(b)$,
and moreover
$\{b^{I}\mid I\in p^{[|b|]}\}$ is an
$\bbF(b,\lambda^{b}(a))^{(p)}$-linear basis of $\bbF(b,\lambda^{b}(a))^{(p)}(b)$
(cf Lemma~\ref{lem:p_linearity}),
which proves {\bf(ii)}.
\end{proof}

\begin{remark}
Lemma~\ref{lem:lambda_calculus}~{\bf(ii)} shows that each $\lambda^{b}_{I}(1/a)$ is a rational function in $b\cup\lambda^{b}(a)$, for any $a\in\pspan[F]{b}$.
\end{remark}

\begin{lemma}["Lambda calculus I"]\label{lem:lambda_calculus_II}
Let $a,b\in\pI[F]$ and let $d\in\ppower[F](a)\cap\ppower[F](b)$ be in the $p$-span of both $a,b$ in $F$.
\begin{enumerate}[{\bf(i)}]
\item
If $b\subseteq\ppower[F](a)$,
then $\lambda^{a}(d)\subseteq\bbF[a,\lambda^{a}(b),\lambda^{b}(d)]$.
\item
If $\ppower[F](a)=\ppower[F](b)$,
then $\bbF(\lambda^{a}(b))=\bbF(\lambda^{b}(a))$.
\item
If $\ppower[F](a)=\ppower[F](b)$,
then $\lambda^{a}(d)\subseteq\bbF(a,\lambda^{b}(a),\lambda^{b}(d))$.
\end{enumerate}
\end{lemma}
\begin{proof}
For {\bf(i)}
we calculate
\begin{align*}
	d
&=
	\sum_{I_{1}\in p^{[|b|]}}b^{I_{1}}\lambda^{b}_{I_{1}}(d)^{p}
\\
&=
	\sum_{I_{1}\in p^{[|b|]}}\big(
	\sum_{I_{2}\in p^{[|a|]}}a^{I_{2}}\lambda^{a}_{I_{2}}(b^{I_{1}})^{p}
	\big)\lambda^{b}_{I_{1}}(d)^{p}
\\
&=
	\sum_{I_{2}\in p^{[|a|]}}
	a^{I_{2}}
	\big(\sum_{I_{1}\in p^{[|b|]}}
	\lambda^{a}_{I_{2}}(b^{I_{1}})
	\lambda^{b}_{I_{1}}(d)
	\big)^{p}.
\end{align*}
Therefore
$\lambda^{a}_{I_{2}}(d)=\sum_{I_{1}\in p^{[|b|]}}\lambda^{a}_{I_{2}}(b^{I_{1}})\lambda^{b}_{I_{1}}(d)$,
for each $I_{2}\in p^{[|a|]}$.
Then
{\bf(i)}
follows from Lemma~\ref{lem:lambda_calculus}~{\bf(i)}.

For {\bf(ii)}:
by the hypothesis, $a$ and $b$ have the same cardinality,
so we may index them both by a cardinal $\fraki=|a|=|b|$:
we write
$a=(a_{\alpha})_{\alpha<\fraki}$
and
$b=(b_{\alpha})_{\alpha<\fraki}$.
Moreover,
both
$A:=\{a^{I}\mid I\in p^{[\fraki]}\}$
and
$B:=\{b^{I}\mid I\in p^{[\fraki]}\}$
are $E:=\pspan[F]C$-linear bases of $E(a)=E(b)$.
Let $M$ be the matrix representing the identity map on
$E(a)=E(b)$
written with respect to the bases
$A$ and $B$.
Writing
$M=(m_{I,J})_{I,J\in p^{[\fraki]}}$,
we have
$m_{I,J}=\lambda_{I}^{b}(a^{J})^{p}$,
where
$a^{J}=\sum_{I\in p^{[\fraki]}}b^{I}\lambda_{I}^{b}(a^{J})^{p}$.
Writing the inverse of $M$ as
$M^{-1}=(\hat{m}_{J,I})_{J,I}$,
we have
$\hat{m}_{J,I}=\lambda_{J}^{a}(b^{I})^{p}$,
where
$b^{I}=\sum_{J\in p^{[\fraki]}}a^{J}\lambda_{J}^{a}(b^{I})^{p}$.
Finally, the coefficients of $M^{-1}$
are contained in the field
generated by the coefficients of $M$.

For {\bf(iii)},
by combining {\bf(i)} and {\bf(ii)},
we have
$$
	\lambda^{a}_{I}(d)
\in
	\bbF[a,\lambda^{a}(b),\lambda^{b}(d)]
\subseteq
	\bbF(a,\lambda^{a}(b),\lambda^{b}(d))
=
	\bbF(a,\lambda^{b}(a),\lambda^{b}(d))
,$$
for each $I\in p^{[|a|]}$.
\end{proof}

In the following, $\mathcal{P}(A)$ denotes the powerset of a set $A$.

\begin{lemma}\label{lem:lambda_calculus_III}
Let $d\subseteq F$
and let
$a,b\in\pI[F]\cap\mathcal{P}(\bbF(d))$
be maximal (with respect to inclusion) among subsets of $\bbF(d)$ that are $p$-independent in $F$.
Then
{\bf(i)} $\bbF(a,\lambda^{a}(d))=\bbF(b,\lambda^{b}(d))$
and
{\bf(ii)} $\rmLambda_{F}^{1}\bbF(d)=\bbF(a,\lambda^{a}(d))$.
\end{lemma}
\begin{proof}
The hypothesis implies that
$d\subseteq \pspan[F]{a}=\pspan[F]{b}$,
so the hypotheses of Lemma~\ref{lem:lambda_calculus_II}~{\bf(i,ii,iii)} are satisfied.
First,
we have
$\lambda^{a}(d)\subseteq\bbF(a,\lambda^{b}(a),\lambda^{b}(d))$,
by Lemma~\ref{lem:lambda_calculus_II}~{\bf(iii)}.
Second, more trivially by Lemma~\ref{lem:p-algebra}~{\bf(i)}, we have
$a\subseteq\bbF[\lambda^{b}(a)]^{(p)}[b]$,
and so certainly
$a\subseteq\bbF[b,\lambda^{b}(a)]$.
Third,
by Lemma~\ref{lem:lambda_calculus}~{\bf(i)},
we have
$\lambda^{b}(a)\subseteq\bbF(b,\lambda^{b}(d))$.
Combining these three observations, we have
$$
	\bbF(a,\lambda^{a}(d))
\subseteq
	\bbF(a,\lambda^{b}(a),\lambda^{b}(d))
\subseteq
	\bbF(b,\lambda^{b}(a),\lambda^{b}(d))
\subseteq
	\bbF(b,\lambda^{b}(d)).
$$
By symmetry of our assumptions on $a$ and $b$,
we have the equality 
$\bbF(a,\lambda^{a}(d))=\bbF(b,\lambda^{b}(d))$
which proves {\bf(i)}.

For the second claim,
by definition,
$\rmLambda^{1}_{F}\bbF(d)$
is the field generated over $\bbF(d)$
by the sets $\lambda^{b'}(d')$,
for all $b'\in\pI[F]\cap\mathcal{P}(\bbF(d))$ and all $d'\in\ppower[F](b')\cap\bbF(d)$.
By hypothesis $a\in\pI[F]\cap\mathcal{P}(\bbF(d))$,
therefore already we have
$\rmLambda^{1}_{F}\bbF(d)\supseteq\bbF(a,\lambda^{a}(d))$.
On the other hand, since $b$ is a maximal subset of $\bbF(d)$ that is $p$-independent in $F$,
and is otherwise arbitrary,
it suffices to show that
$\bbF(b,\lambda^{b}(d'))\subseteq\bbF(a,\lambda^{a}(d))$.
By Lemma~\ref{lem:lambda_calculus}~{\bf(ii)},
$\lambda^{b}(d')\subseteq\bbF(b,\lambda^{b}(d))$;
and combining this with {\bf(i)}
we have
$\bbF(b,\lambda^{b}(d'))\subseteq\bbF(b,\lambda^{b}(d)\subseteq\bbF(a,\lambda^{a}(d))$,
as required.
\end{proof}

We also give the following version of Lemma~\ref{lem:lambda_calculus_II} for relative $p$-independence.
As before, we denote the concatenation of well-ordered sets $a$ and $b$ by $\concat{a}{b}$.
We allow an exception to this convention in superscripts,
where for want of space we denote concatenation by juxtaposition $ab$.
In such circumstances there is no risk of confusion with multiplication.
The (unordered) set underlying $\concat{a}{b}$ is just the union $a\cup b$.

\begin{lemma}["Lambda calculus II"]\label{lem:lambda_calculus_IIc}
Let $a,b\in\pI[(F/C)]$,
let $c\in\pB[C]$,
and let $d\in\ppower[F]C(a)\cap\ppower[F]C(b)$ be in the $p$-span of both $a,b$ in $F$ over $C$.
\begin{enumerate}[{\bf(i)}]
\item
If $b\subseteq\ppower[F]C(a)$,
then $\lambda^{ca}(d)\subseteq C[\lambda^{cb}(d),\lambda^{ca}(b),a]$.
\item
If $\ppower[F]C(a)=\ppower[F]C(b)$,
then $C(\lambda^{ca}(b))=C(\lambda^{cb}(a))$.
\item
If $\ppower[F]C(a)=\ppower[F]C(b)$,
then $\lambda^{ca}(d)\subseteq C(\lambda^{cb}(a),\lambda^{cb}(d),a)$.
\end{enumerate}
\end{lemma}
\begin{proof}
The general hypothesis implies that both $\concat{c}{a}$ and $\concat{c}{b}$ are $p$-independent in $F$,
and that $d\in\ppower[F](a,c)\cap\ppower[F](b,c)$.
For {\bf(i)}:
the hypothesis implies that 
$\concat{c}{b}\subseteq\pspan[F]{a,c}$.
By Lemma~\ref{lem:lambda_calculus_II}~{\bf(i)} we have
$$\lambda^{ca}(d)\subseteq\bbF[\lambda^{cb}(d),\lambda^{ca}(c\cup b),a,c]\subseteq C[\lambda^{cb}(d),\lambda^{ca}(b),a].$$
For {\bf(ii)}:
the hypothesis implies that 
$\pspan[F]{a,c}=\pspan[F]{b,c}$.
By Lemma~\ref{lem:lambda_calculus_II}~{\bf(ii)} we have
$\bbF(\lambda^{ca}(c\cup b))=\bbF(\lambda^{cb}(c\cup a))$,
which yields
$C(\lambda^{ca}(b))=C(\lambda^{cb}(a))$.
For {\bf(iii)}:
the hypothesis again implies that 
$\pspan[F]{a,c}=\pspan[F]{b,c}$.
By 
{\bf(i)} and {\bf(ii)},
as well as 
by Lemma~\ref{lem:lambda_calculus_II}~{\bf(iii)},
we have
$
	\lambda^{ca}(d)
\subseteq
	\bbF(\lambda^{cb}(a),\lambda^{cb}(c),\lambda^{cb}(d),a,c)
\subseteq
	C(\lambda^{cb}(a),\lambda^{cb}(d),a)
$.
\end{proof}

...and we give the following version of Lemma~\ref{lem:lambda_calculus_III} for relative $p$-independence.

\begin{lemma}\label{lem:lambda_calculus_IIIc}
Let $d\subseteq F$,
let $a,b\in\pI[(F/C)]\cap\mathcal{P}(C(d))$ be maximal (with respect to inclusion) among subsets of $C(d)$ that are $p$-independent in $F$ over $C$,
and let $c\in\pB[C]$.
Then
{\bf(i)}
$C(a,\lambda^{ca}(d))=C(b,\lambda^{cb}(d))$
and
{\bf(ii)}
$\rmLambda_{F}^{1}C(d)=C(a,\lambda^{ca}(d))$.
\end{lemma}
\begin{proof}
For {\bf(i)} of course we apply Lemma~\ref{lem:lambda_calculus_III}~{\bf(i)} to $\concat{c}{a}$ and $\concat{c}{b}$:
we have
$\bbF(a,c,\lambda^{ca}(d))=\bbF(c,b,\lambda^{cb}(d))$,
from which we deduce
$C(a,\lambda^{ca}(d))=C(b,\lambda^{cb}(d))$.
For {\bf(ii)},
the concatenation $\concat{c}{a}$ is indeed $p$-independent in $F$, and is maximal among such subsets of $C(d)$, with respect to inclusion.
Thus, by Lemma~\ref{lem:lambda_calculus_III}~{\bf(ii)}, we have
$\rmLambda^{1}_{F}C(d)=\bbF(a,c,\lambda^{ca}(C\cup d))=C(a,\lambda^{ca}(d))$,
since $F/C$ is separable.
\end{proof}

\subsection{The local lambda closure}
\label{section:local_lambda_closure}

Given a field extension $F/C$ and
a well-ordered subset $a$ of $F$,
we denote by $\pInd{F/C}(a)$ the well-ordered maximal subset of $a$ which is $p$-independent in $F$ over $C$,
taken from the left.
More precisely,
writing the well-ordering of $a$ as $(a_{\alpha})_{\alpha<\mu}$,
we obtain
$\pInd{F/C}(a)$
recursively by adding each element if (and only if) it is $p$-independent over $C$ together with what has already been added,
and we equip it with the well-order induced from $a$.
Let $\bigL(F/C)$ denote the set
of pairs $(a,b)$,
where both $a,b$ are well-ordered subsets of $F$ and $a\in\pI[(F/C)]$.
For the remainder of this section we suppose that $F/C$ is separable.
We define a family of operations $\pL{F/c}$ on $\bigL(F/C)$.

\begin{definition}[{``Splitting pairs''}]\label{def:little_lambda}
Let
$c\in\pB[C]$.
Given $(a,b)\in\bigL(F/C)$ we let $\pL{F/c}(a,b)=(a',b')$, where
\begin{enumerate}[{\bf(i)}]
\item
$a'$
is the concatenation
$\concat{a}{\pInd{F/C(a)}(b)}$,
and
\item
$b'$
is the concatenation
$\concat{b}{\lambda^{ca'}(b)}$,
where $\lambda^{ca'}(b)$ is ordered by the lexicographic order determined by $b\times p^{[|ca'|]}$.
\end{enumerate}
\end{definition}

This definition is illustrated in Figure~\ref{diag:splitting}.
Note that indeed $a'\in\pI[(F/C)]$ and $b\in\ppower[F]C(a')$,
otherwise $a'$ could be properly extended within $b$, contradicting maximality.
Thus $b'$ is well-defined, and moreover $\pL{F/c}(a,b)\in\bigL(F/C)$,
so $\pL{F/c}$ does define an operation on $\bigL(F/C)$.

 \begin{figure}[ht]
 \centering
 \begin{tikzpicture}[auto,scale=1.50,thick,>=latex,block/.style={draw, fill=white, rectangle, minimum height=3em, minimum width=6em},
	 greybox/.style={rectangle, rounded corners=10pt, minimum width=2.30cm, minimum height=1.00cm, text centered, draw=black, fill=black!02},
	 thingreybox/.style={rectangle, rounded corners=05pt, minimum width=2.15cm, minimum height=0.50cm, text centered, draw=black, fill=black!02},
	 roundedgreybox/.style={rectangle, rounded corners=12pt, minimum width=4.50cm, minimum height=0.85cm, text centered, draw=black, fill=black!02},
	 narrowroundedgreybox/.style={rectangle, rounded corners=12pt, minimum width=2.80cm, minimum height=1.00cm, text centered, draw=black, fill=black!02},
	 chunkygreybox/.style={rectangle, rounded corners=05pt, minimum width=3.50cm, minimum height=0.50cm, text centered, draw=black, fill=black!02},
 pinkbox/.style={rectangle, rounded corners=08pt, minimum width=1.5cm, minimum height=0.6cm, text centered, draw=black, fill=red!10},
 bluebox/.style={rectangle, rounded corners, minimum width=2.8cm, minimum height=0.8cm, text centered, draw=black, fill=orange!10},
 orangebox/.style={rectangle, rounded corners, minimum width=2.8cm, minimum height=0.8cm, text centered, draw=black, fill=yellow!10},
 thmarrow/.style={ultra thick,-{Stealth}},
 hyparrow/.style={dashed,thick,-{Stealth}},
 R4arrow/.style={dotted,thick,-{Stealth}},
 myarrow/.style={thick,-{Stealth}},
 ]
 
 \coordinate (0) at (0,0); 
 \coordinate (1) at (-3.5,0); 
 \coordinate (1a) at (-2.5,0);
 \coordinate (2) at (-7.0,0);
 \coordinate (2a) at (-6.0,0);
 
 \coordinate (a) at ($(0)+(0,0)$);
 \coordinate (a1) at ($(a)+(0,0)$);
 \coordinate (a2) at ($(a)+(3,0)$);
 
 \coordinate (b) at ($(0)+(0,+1.50)$);
 \coordinate (b1) at ($(b)+(0,0)$);
 \coordinate (b2) at ($(b)+(3,0)$);
 
 \coordinate (a3) at ($(a2)+(-0.15,0)$);
 
 \node[thingreybox] (A1) [] at (a1) {$a$};
 \node[narrowroundedgreybox] (A2) [] at (a2) {};
 \node[] (A2a) [] at ($(a2)+(+0.50,+0.00)$) {$b$};
 
 \node[chunkygreybox] (B1) [] at (b1) {$a'=\concat{a}{\pInd{F/C(a)}(b)}$};
 \node[roundedgreybox] (B2) [] at (b2) {$b'=\concat{b}{\lambda^{ca'}(b)}$};
 
 \draw[thmarrow] (a3) to node[below] {} (B1);
 \draw[thmarrow] (A1) to node[below] {} (B1);
 \draw[thmarrow] (A2) to node[right] {$\lambda^{ca'}$} (B2);
 
 \node[pinkbox] (A3) [] at (a3) {$\pInd{F/C(a)}(b)$};
 \end{tikzpicture}
 \caption{The ``splitting pairs'' map $(a,b)\mapsto\pL{F/c}(a,b)$}
 \label{diag:splitting}
 \end{figure}
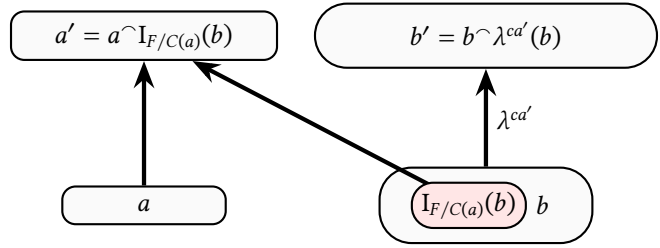

\begin{remark}\label{rem:splitting}
We emphasize that $\pL{F/c}(a,b)$ is obtained from $(a,b)$ simply by repartitioning and reordering, and just one application of $\lambda^{ca'}$.
\end{remark}

Denote by $\preceq_{1}$ the partial order on well-ordered subsets of $F$ by writing $a\preceq_{1}a'$ if $a$ is an initial segment of $a'$,
and let $\preceq$ be the partial order on $\bigL(F/C)$ given by $(a,b)\preceq(a',b')$ if and only if
both $a\preceq_{1}a'$ and $b\preceq_{1}b'$.
In the following lemma we continue to write
$\pL{F/c}(a,b)=(a',b')$.

\begin{lemma}\label{lem:chain_1}
For any
$(a,b)\in\bigL(F/C)$
we have
$(a,b)\preceq\pL{F/c}(a,b)$.
Moreover
$a'$ $p$-spans $C(a,b)$ in $F$ over $C$,
i.e.~$C(a,b)\subseteq\ppower[F]C(a')$,
and 
$\rmLambda_{F}^{1}C(a,b)=C(a',\lambda^{ca'}(b))=C(\pL{F/c}(a,b))$.
\end{lemma}
\begin{proof}
Inspecting the definition of $\pL{F/c}$,
it is clear that $(a,b)\preceq(a',b')$
and $b\subseteq C(a',\lambda^{ca'}(b))$.
This shows that
$C(a',\lambda^{ca'}(b))=C(a',b')=C(\pL{F/c}(a,b))$.
By Lemma~\ref{lem:lambda_calculus_IIIc}~{\bf(ii)},
$\rmLambda_{F}^{1}C(a,b)=C(a',\lambda^{ca'}(b))$.
\end{proof}

\begin{definition}\label{def:partial_lambda}
For a well-ordered subset
$a$
of $F$,
we let
$(\pl{F/c}^{n}(a))_{n<\omega}$
be the sequence given recursively by
$\pl{F/c}^{0}(a)=(\emptyset,a)$ and
$\pl{F/c}^{n+1}(a)=\pL{F/c}(\pl{F/c}^{n}(a))$.
We also write
$\pl{F/c}^{n}(a)=(\pl{F/c}^{n,\pII}(a),\pl{F/c}^{n,\pSS}(a))$.
\end{definition}

\begin{lemma}\label{lem:chain_2}
For a well-ordered subset
$a$
of $F$,
we have
$\pl{F/c}^{n}(a)\preceq\pl{F/c}^{n+1}(a)$.
Moreover
$\pl{F/c}^{n+1,I}(a)$ $p$-spans $C(\pl{F/c}^{n}(a))$ over $C$
and
$\rmLambda^{1}_{F}C(\pl{F/c}^{n}(a))=C(\pl{F/c}^{n+1}(a))$.
\end{lemma}
\begin{proof}
This follows immediately from Lemma~\ref{lem:chain_1}.
\end{proof}

\begin{definition}\label{def:lambda_main}
Let
$\lambda_{F/c}(a)=(\lambda_{F/c}^{\pII}a,\lambda_{F/c}^{\pSS}a)$
be the direct limit of this chain of pairs
$(\pl{F/c}^{n}(a))_{n<\omega}$,
where $\pl{F/c}^{n}(a)=(\pl{F/c}^{n,\pII}(a),\pl{F/c}^{n,\pSS}(a))$,
and write
$\lambda_{F/c}a=\lambda_{F/c}^{\pII}a\cup\lambda_{F/c}^{\pSS}a$.
We call $\lambda_{F/c}a$
the
{\em local lambda closure}
of $a$, with respect to $c$.
\end{definition}

Observe that
$\lambda_{F/c}(a)\in\bigL(F/C)$,
since $\bigL(F/C)$ is closed under unions of chains with respect to $\preceq$.

\begin{proposition}\label{prp:chain_3}
For a well-ordered subset
$a$
of $F$,
we have
\begin{enumerate}[{\bf(i)}]
\item
$\rmLambda_{F}C(a)=C(\lambda_{F/c}a)$,
\item
$\lambda_{F/c}^{\pII}a$
is a $p$-basis
of both
$C(\lambda_{F/c}^{\pII}a)$
and
$C(\lambda_{F/c}a)$,
\item
$C(\lambda_{F/c}a)/C(\lambda_{F/c}^{\pII}a)$ is separated,
and
\item
both
$F/C(\lambda_{F/c}^{\pII}a)$ 
and
$F/C(\lambda_{F/c}a)$
are separable.
\end{enumerate}
\end{proposition}
\begin{proof}
By induction and Lemma~\ref{lem:chain_1},
we have 
$\rmLambda_{F}^{n}C(a)=C(\pl{F/c}^{n}(a))$,
for all $n<\omega$.
Then
{\bf(i)}
follows from
Lemma~\ref{lem:big_lambda}.
For {\bf(ii)}:
$\lambda_{F/c}^{\pII}a$ is a directed union of sets $p$-independent in $F$ over $C$,
so it also $p$-independent in $F$ over $C$.
This implies already that it is a $p$-basis of $C(\lambda_{F/c}^{\pII}a)$ over $C$.
It remains to see that
$\lambda_{F/c}^{\pSS}a\subseteq\ppower[C(\lambda_{F/c}^{\pSS}a)](\lambda_{F/c}^{\pII}a)$.
But $\lambda_{F/c}^{\pSS}a$ is then directed union of $\pl{F/c}^{n,\pSS}(a)$
and $\rmLambda_{F}^{1}C(\pl{F/c}^{n}(a))\subseteq C(\pl{F/c}^{n+1}(a))$.
This shows that $\lambda_{F/c}^{\pII}a$ is a $p$-basis of $C(\lambda_{F/c}a)$ over $C$.
Together with Lemma~\ref{lem:Mac_Lane_II},
this also proves {\bf(iii)} and {\bf(iv)}.
\end{proof}

\begin{lemma}\label{lem:finite}
For a well-ordered subset $a$ of $F$, and a finite subset $b\subseteq F$,
there exists $n_{1}\in\mathbb{N}$ such that $b$ is separable over $C(\pl{F/c}^{n_{1}}(a))$.
If $a$ is also finite, then there is $n_{2}\in\mathbb{N}$ such that $b$ is separable over $C(\pl{F/c}^{n_{2}}(a))$,
where $\pl{F/c}^{n_{2}}(a)$ is constructed with respect to any well-ordering on $a$.
\end{lemma}
\begin{proof}
First note that there is $m\in\mathbb{N}$ such that
$\mathrm{trdeg}(b/C(\pl{F/c}^{m}(a)))=\mathrm{trdeg}(b/C(\lambda_{F/c}a))$.
Let $b_{1}\subseteq b$ be a separating transcendence basis of $b$ over $C(\lambda_{F/c}(a))$.
Now choose $n_{1}\geq m$ such that
$[C(\pl{F/c}^{n_{1}}(a),b):C(\pl{F/c}^{n_{1}}(a),b_{1})]=[C(\lambda_{F/c}a,b):C(\lambda_{F/c}a,b_{1})]$.
It follows that $C(\pl{F/c}^{n_{1}}(a),b)$ is separably algebraic over $C(\pl{F/c}^{n_{1}}(a),b_{1})$.
We let $n_{2}$ be the maximum value of $n_{1}$,
taken across the finitely many different well-orderings of $a$.
\end{proof}

\begin{definition}\label{def:finite}
Let $a,b$ be two finite subsets of $F$.
By the previous lemma, we may choose $n\in\mathbb{N}$
minimal such that
$b$ is separable over $C(\pl{F/c}^{n}(a))$,
taken with respect to any well-ordering of $a$.
Denote
$\lambda_{F/b/c}a:=\pl{F/c}^{n,I}(a)\cup\pl{F/c}^{n,S}(a)$.
\end{definition}

\begin{remark}\label{rem:sigma}
The meaning of Lemma~\ref{lem:finite} is:
when $a,b$ are both finite subsets of $F$,
say $(a,b)\in F^{m+n}$,
there is a bound on the depth of the recursive process required in the construction of 
the minimal extension $C(\lambda_{F/b/c}a)$ of $C(a)$ over which $b$ is separable, 
when this processes is effected with respect to any well-ordering of $a$, and with respect to a fixed $c\in\pB[C]$.
Moreover, the resulting well-ordered set $\lambda_{F/b/c}a$
is a finite tuple.
We denote
$\ell:=|\lambda_{F/b/c}a|\in\mathbb{N}$.
Then $a$ is the image of $\lambda_{F/b/c}a$ under a coordinate projection
$\sigma_{F/b/c}:\bbA_{C}^{\ell}\rightarrow\bbA_{C}^{m}$.
In particular, this applies when $F/C(a)$ is finitely generated.
\end{remark}

\begin{remark}\label{rem:lambda_2}
Let $F$ be any field.
For any subfield $C\subseteq F$
we have
$|\rmLambda_{F}C|=|C|$.
To see this, we first note that if $C$ is finite then $C$ is perfect, so already $\rmLambda_{F}C=C$.
Now, let $a$ be a well-ordered generating set for $C$.
Then $\lambda_{F/\emptyset}a$ is a countable direct limit of tuples
$\pl{F/\emptyset}^{n}(a)$,
and
we see that there is a finite-to-one map from
$\lambda^{a'}(b)$
to $b$,
using the notation of Definition~\ref{def:little_lambda}.
Thus 
$|\pl{F/\emptyset}^{n}(a)|\leq\aleph_{0}\cdot|a|\leq|C|$.
Arguing inductively,
we have
$|\lambda_{F/\emptyset}c|\leq|C|$,
and thus $|\rmLambda_{F}C|=|C|$.
This should be compared with \cite[Lemma 8]{A19} in which it was shown that $|\rmLambda_{F}C|\leq\aleph_{0}$ for $|C|\leq\aleph_{0}$.
\end{remark}

\subsection{The lambda language}
\label{section:lambda_language}

The {\em elementary imperfection degree} (or {\em Ershov degree/invariant})
of a field extension $F/C$ is
$\Impdeg(F/C)=\impdeg(F/C)$ if 
$\impdeg(F)$ is finite;
or it is symbolically infinite,
i.e.~$\Impdeg(F/C)=\infty$,
if $\impdeg(F)$ is infinite.
As with the usual imperfection degree,
we write $\Impdeg(F):=\Impdeg(F/\bbF)$ for the absolute case.
We are only ever going to consider $\Impdeg(F/C)$ when $F/C$ is separable,
in which case $\impdeg(F)=\impdeg(F/C)+\impdeg(C)$.

There are several common languages used to study imperfect fields,
usually construed as expansions of $\Lring$.
When dealing with fields of finite (or, at least, bounded) imperfection degree $\fraki\in\mathbb{N}$,
it may suffice first to expand the language by constants for a choice of $p$-basis,
obtaining 
$\Lbasis=\Lring\cup\{b_{1},\ldots,b_{\fraki}\}$,
then as a second step further adjoining function symbols for the elements of $\lambda^{b}$
to get
$\Lblambda=\Lring\cup\{b_{1},\ldots,b_{\fraki}\}\cup\{\lambda^{b}_{I}\mid I\in p^{[\fraki]}\}$.
The disadantages of this approach are clear:
it artificially distinguishes a $p$-basis and
it depends on both $p$ and $\fraki$.
Moreover,
when the imperfection degree is infinite this approach breaks down:
if we expand a field $F$ be constants for an infinite $p$-basis,
there are certainly elementary extensions $F\preceq F^{*}$ in which those constants no longer form a $p$-basis.

Another common approach is to adjoin $n$-ary relation symbols,
which are to be interpreted as the relation of $p$-independence,
i.e.~%
$\LQ=\Lring\cup\{Q_{n}(x_{1},\ldots,x_{n})\mid n<\omega\}$.
This language is better suited to the study of infinite (or unbounded) imperfection degree,
and it does not artificially distinguish one $p$-basis.
In the corresponding second step, we might then adjoin
{\em all}
of the lambda maps
corresponding to any $p$-independent tuple:
these are essentially the parameterized lambda functions of
Definition~\ref{def:lambda}.
This potentially vast collection of maps gives rise to syntactically complex terms,
with compositions of $\Lring$-terms and
deeply nested lambda maps,
with respect to differing $p$-independent tuples.
This formal complexity hides a much simpler structure.

\begin{definition}\label{def:language_p}
For $p$ a prime number,
we let 
$\LlambdaP$
be the language with signature
$$\{l_{I}(x,\underline{y})\mid I\in p^{[n]},\underline{y}=(y_{i})_{i<n},n<\omega\}$$
consisting of a family of $(1+n)$-ary function symbols indexed by $p^{[n]}$,
with variables $(x,y_{0},\ldots,y_{n-1})$, for $n<\omega$ a natural number.
Any field $F$ of characteristic $p$,
which may be enriched already with additional structure,
can be viewed naturally as an $\LlambdaP$-structure:
we let the interpretation of $l_{I}$ be the function $\lambda_{I}$,
extended to be zero where it was not before defined:
\begin{align*}
(a,b)\longmapsto\left\{
	\begin{array}{ll}
	\lambda^{b}_{I}(a)&a\in\pspan[F]{b},b=(b_{i})_{i<n}\in\pI[F]
	\\
	0&\text{else},
	\end{array}
\right.
\end{align*}
for each $I\in p^{[n]}$.
If $F$ is an $\Lang$-structure, expanding a field of characteristic $p$,
then $\bar{F}$ will denote the natural $\Lang\cup\LlambdaP$-expansion of $F$.
\end{definition}

For each $p\in\mathbb{P}$,
let $\chi_{p}$ be the $\Lring$-sentence
\begin{align*}
	\underbrace{1+\ldots+1}_{\text{$p$ times}}=0,
\end{align*}
and,
for each $\fraki\in\mathbb{N}$,
let $\iota_{p,\leq\fraki}$ be the $\Lring$-sentence
\begin{align*}
\exists b=(b_{0},\ldots,b_{\fraki-1})\;
\forall a\;
\exists y=(y_{I})_{I\in p^{[\fraki]}}\;:\;
x=\sum_{I\in p^{[\fraki]}}b^{I}y_{I}^{p},
\end{align*}
and let $\iota_{p,\fraki}$ be the $\Lring$-sentence
\begin{align*}
\left\{
	\begin{array}{ll}
		\iota_{p,\leq0}
		&
		\text{if $\fraki=0$},
\\
		(\iota_{p,\leq\fraki}\wedge\neg\iota_{p,\leq\fraki-1})
		&
		\text{if $\fraki>0$}.
	\end{array}
\right.
\end{align*}

\begin{definition}\label{def:XF}
For $p\in\mathbb{P}\cup\{0\}$ we define
\begin{align*}
	\Xth_{p}
&:=
\left\{
\begin{array}{ll}
	\{\chi_{p}\}
&
	\text{if $p>0$, and}
\\
	\{\neg\chi_{\ell}\mid\ell\in\mathbb{P}\}
&
	\text{if $p=0$}.
\end{array}
\right.
\end{align*}
For $p\in\mathbb{P}$ and $\frakI\in\mathbb{N}\cup\{\infty\}$ we define
\begin{align*}
	\Xth_{p,\frakI}
&:=
\left\{
\begin{array}{ll}
	\Xth_{p}\cup\{\iota_{p,\frakI}\}
&
	\text{if $\frakI<\infty$, and}
\\
	\Xth_{p}\cup\{\neg\iota_{p,\leq\fraki}\mid\fraki\in\mathbb{N}\}
&
	\text{if $\frakI=\infty$.}
\end{array}
\right.
\end{align*}
Let
\begin{align*}
\begin{array}{cllc}
	\Fth_{p,\frakI}
&:=&
	\Fth\cup\Xth_{p,\frakI}
	&
	\text{for $(p,\frakI)\in\mathbb{P}\times(\mathbb{N}\cup\{\infty\})$, and}
\\
	\Fth_{0}
&:=&
	\Fth\cup\Xth_{0}.
\end{array}
\end{align*}
Then $\Fth_{p,\frakI}$ is the $\Lring$-theory of fields of characteristic $p>0$
and of elementary imperfection degree $\frakI$,
and $\Fth_{0}$
is the theory of fields of characteristic zero.
These subscripts will be similarly applied to other theories of fields $T$
in languages $\Lang\supseteq\Lring$:
we write
$T_{0}=T\cup\Xth_{0}$, 
$T_{p}=T\cup\Xth_{p}$, 
and
$T_{p,\frakI}=T\cup\Xth_{p,\frakI}$, 
for $(p,\frakI)\in\mathbb{P}\times(\mathbb{N}\cup\{\infty\})$.
\end{definition}

These sentences
and theories
will be important
to our work on theories of separably tame valued fields,
see~\ref{section:resplendent_STVF}.

\begin{definition}\label{def:theories_p}
Let $\FpL$ be the $\LringlambdaP$-theory extending $\Fth_{p}$
by axioms that ensure the new function symbols are
interpreted suitably as the parameterized lambda maps.
For any $\Lring$-theory $T$ of fields of characteristic $p$,
we let $T_{\lambda(p)}:=T\cup\FpL$ be the natural $\LringlambdaP$-theory extending $T$.
\end{definition}

\begin{fact}\label{fact:separable_maps}
We have already pointed out that each $F\models\Fth_{p}$ admits a natural expansion $\bar{F}$ to an $\LringlambdaP$-structure.
In fact, $\bar{F}$ is the unique $\LringlambdaP$-structure expanding $F$ to a model of $\FpL$,
and the extra structure of $\bar{F}$ is $\Lring$-definable in $F$.
A ring morphism $\varphi:E\rightarrow F$ between fields of characteristic $p$ is separable
if and only if
$\varphi$ is an $\LringlambdaP$-embedding $\bar{E}\rightarrow\bar{F}$.
This gives an isomorphism of categories between the category of fields of characteristic $p$, equipped with separable field embeddings,
and the category of $\LringlambdaP$-structures that are models of $\FpL$, equipped with $\LringlambdaP$-embeddings.
\end{fact}

\begin{remark}
The language $\LlambdaP$
is suitable for studying fields of infinite imperfection degree,
and even allows a uniform approach to fields of varying elementary imperfection degree.
Nevertheless,
if we are willing to focus on fields of imperfection degree bounded by some {\em finite} $\fraki\in\mathbb{N}$,
we may instead work with a more closely adapted sublanguage
$\LlambdaPi\subseteq\LlambdaP$
which has the signature
$$\{l_{I}(x,\underline{y})\mid I\in p^{[\fraki]},\underline{y}=(y_{0},\ldots,y_{\fraki-1})\},$$
consisting of those $(1+\fraki)$-ary function symbols from $\LlambdaP$ that are indexed by $I\in p^{[\fraki]}$.
Any field $F$ of characteristic $p$
can be viewed naturally as an $\LlambdaPi$-structure
by taking the reduct of the natural $\LlambdaP$-structure.
Moreover, for many purposes,
if $T$ is a theory of fields of characteristic $p$ with imperfection degree bounded by some $\fraki\in\mathbb{N}$,
the role played by $T_{\lambda(p)}$
may be adequately played by its 
$\LringlambdaPi=\Lring\cup\LlambdaPi$-reduct.
\end{remark}

It is also possible to remove the dependence on $p$ in the languages and theories introduced above,
albeit rather inelegantly.

\begin{definition}[Uniformity in $p$]\label{def:language}
Let $\Llambda$ be the language
with signature
\begin{align*}
	\{l_{p,I}(x,\underline{y})\mid p\in\mathbb{P},I\in p^{[n]},\underline{y}=(y_{i})_{i<n},n<\omega\},
\end{align*}
which amounts to the disjoint union of the signatures of $\LlambdaP$, for $p$ any prime number.
Now, any field $F$ may be expanded to an $\Lang\cup\Llambda$-structure $\tilde{F}$:
let the interpretation of $l_{p,I}$ be the following function
\begin{align*}
(a,b)\longmapsto\left\{
	\begin{array}{ll}
	\lambda^{b}_{I}(a)&a\in\pspan[F]{b},b=(b_{i})_{i<n}\in\pI[F],p=\mathrm{char}(F),
	\\
	0&\text{else}.
	\end{array}
\right.
\end{align*}
Thus if $F$ is of characteristic zero, all these function symbols are interpreted by the zero function.
\end{definition}

\begin{definition}\label{def:theories}
We let $\FL$ be the $\Lringlambda=(\Lring\cup\Llambda)$-theory extending $\Fth$ by axioms that ensure the new function symbols are suitably interpreted as the parameterized lambda maps, when $p$ is the (positive) characteristic, and as the zero function, otherwise.
In this way we may naturally extend any theory $T$ of fields (possibly enriched to $\Lang$-structures for some $\Lang\supseteq\Lring$) to a theory $T_{\lambda}$ of $(\Lang\cup\Llambda)$-structures.
\end{definition}

\begin{fact}\label{fact:separable_maps_II}
The analogue of Fact~\ref{fact:separable_maps} applies to $\Llambda$:
each $F\models\Fth$ admits a natural expansion $\tilde{F}\models\Th(F)_{\lambda}$ to an $\Lringlambda$-structure,
and the extra structure of $\tilde{F}$ is $\Lring$-definable in $F$.
\end{fact}

\begin{remark}[Separably closed fields]\label{rem:SCF}
As mentioned above,
the languages $\LringlambdaP$ are not new,
and have been used to study the theory of separably closed fields,
alongside the languages $\Lbasis$ and $\LQ$.
Let
$\SCF$
be the theory of
separably closed
fields
in the language
$\Lring$
of rings.
In \cite{Ershov},
Ershov showed that the completions of $\SCF$ are
$\SCF_{0}:=\ACF_{0}$
and
$\SCF_{p,\frakI}$
for $p\in\mathbb{N}$ and $\frakI\in\mathbb{N}\cup\{\infty\}$.
Later Wood, in \cite{Wood}, showed that
each $\SCF_{p,\frakI}$ is stable, not superstable.
In \cite{Delon82},
Delon showed that
$\SCF$
is model complete in $\Lbasis$, when the imperfection degree is finite,
and in $\LQ$ in general.
Moreover, Delon showed that $\SCF$ admits quantifier elimination in $\LringlambdaP$,
Still later, 
\cite{CCSSW}
studied types in $\SCF$ using the language $\LringlambdaP$.
As already mentioned,
Hong used a variant of this language to prove quantifier elimination for $\SCVF_{p,\frakI}$ in $\Lvlambda$,
see \cite{Hong}.
In a different direction, Srour
has shown
in \cite{Srour}
that,
$\SCF_{p,\frakI}$ is equational in $\Lblambda$,
for $\frakI<\infty$.
\end{remark}

\begin{proof}[{Proof of Theorem~\ref{thm:intro_1}}]
We recall the setting of the theorem:
$F/C$ is a separable field extension,
$c\in\pB[C]$,
and $a$ is any well-ordered subset of $F$.
In Definition~\ref{def:lambda_main}
we already have given the definition of
$\lambda_{F/c}a$
as the union of two well-ordered subsets of $F$,
$\lambda_{F/c}^{I}a$
and
$\lambda_{F/c}^{S}a$,
where
the former is $p$-independent in $F$,
and the latter is a subset of
$C(\lambda_{F/c}^{S}a)^{(p)}(\lambda_{F/c}^{I}a)$.

The first claim {\bf(i)}
was proved in Proposition~\ref{prp:chain_3}~{\bf(i)},
and we saw in Lemma~\ref{lem:finite}
that,
when $F=C(a,b)$ is finitely generated over $C$,
we may replace
$\lambda_{F/c}a$
with a finite tuple
$\lambda_{F/b/c}a$,
as defined in Definition~\ref{def:finite},
which proves {\bf(iii)}.

For {\bf(ii)}, we observe that the recursive construction of $\lambda_{F/c}a$,
given in Definition~\ref{def:little_lambda},
constructs $a'$ and $b'$ using simply repartitioning of tuples, and one application of the function $\lambda^{ca'}$ to $b$.
Thus means that $b'$ is the union of $b$ and the interpretation of $l_{p,I}(b_{i},\concat{c}{a'})$, for elements $b_{i}$ of $b$ and for $I\in p^{[n]}$ for $n=|\concat{c}{a'}|$.
In particular, $b'$ is formed from $\Llambda$-terms in the elements of the tuples $a'$, $b$, and $c$.
\end{proof}

\section{T-henselianity, the Implicit Function Theorem, and existentially definable sets}
\label{section:IFT}

In this section we study sets defined by existential formulas in the language $\Lring$ of rings,
allowing parameters,
in fields equipped with a henselian topology,
in the sense of Prestel and Ziegler
\cite{PZ78},
though in that paper such topological fields are called ``$t$-henselian'' and the topologies are identified with any corresponding filters of neighbourhoods of $0$.
To this end,
throughout this section
we suppose the following:
\begin{enumerate}
\item[\DAG]
\label{DAG}
{\bf $K/C$ is a separable field extension
and $K$ admits a henselian topology
$\tau$}.
\end{enumerate}
Note that henselian topologies are in particular ``$V$-topologies'',
which means that they are induced by valuations or absolute values,
by \cite{KowalskyDurbaum,Fleischer}.
If $K$ is not separably closed then it admits at most one henselian topology;
moreover, if $\tau$ is such a topology on $K$, then there is a uniformly existentially $\Lring$-definable family of subsets of $K$ that forms a basis for the filter of $\tau$-neighbourhoods of $0$,
see e.g.~\cite[(7.11) Remark]{PZ78}, as corrected in \cite[Lemma]{Prestel}.

\begin{remark}\label{rem:motivation}
In algebraically closed fields $K$, for example in $K=\mathbb{C}$ or $\bbF_{p}^{\mathrm{alg}}$,
all infinite definable subsets of $1$-dimensional affine space $\bbA^{1}$ (i.e.~of $K$) are cofinite,
and so are in particular Zariski open.
In real closed fields, for example in $K=\mathbb{R}$, 
all infinite definable subsets of $\bbA^{1}$ are finite unions of intervals,
at least one of which must be nontrivial.
Thus they have nonempty interior in the order topology.
In $p$-adically closed fields, for example in the fields $K=\mathbb{Q}_{p}$ of $p$-adic numbers and finite extensions thereof,
all infinite definable subsets of $\bbA^{1}$ have nonempty interior in the $p$-adic topology.
Macintyre, in his survey article \cite{Macintyre},
comments on the fact that all infinite definable subsets of $\bbA^{1}$ in local fields of characteristic zero have nonempty interior,
implicitly raising the same question for local fields of positive characteristic.
For such a field $K=\ps{\bbF_{q}}$
it is clear that $K^{(p)}$ is an infinite definable subset with empty interior in the $t$-adic topology.
Thus any reasonable analogue of Macintyre's observation for local fields of positive characteristic must take into account at least additive and multiplicative cosets of
$t^{p^{n}}$-adically open subsets of $K^{(p^{n})}$, taken inside $K$.
However, as the following example shows, even these sets are not rich enough.
\end{remark}

\begin{example}\label{ex:bad_interior}
There is an existentially $\Lval(t)$-definable subset of $K=\ps{\bbF_{p}}$
which is neither an additive coset nor a multiplicative coset in $K$ of a subset of $K^{(p^{n})}$ with nonempty interior in the $t^{p^{n}}$-adic topology, for any $n<\omega$.
For example, if $p\neq2$, the set $X=\{x^{p}+tx^{2p}\mid x\in K\}$ is existentially $\Lring(t)$-definable,
and yet it is easy to see that the sets $aX+b$, for $a\neq 0$, have empty $t^{p}$-adic interior.
\end{example}

For any constructible set $U\subseteq\bbA_{C}^{n}$,
i.e.~a Boolean combination of subvarieties of $\bbA_{C}^{n}$ defined over $C$,
we denote by $U(K)$ the set of $K$-rational points of $U$.
For a set (or tuple) $f\subseteq C[X_{0},\ldots,X_{n-1}]$ of polynomials, let $Z(f)\subseteq\bbA_{C}^{n}$ be the affine variety defined by the vanishing of $f$.

\begin{definition}\label{def:locus}
For an $m$-tuple $a$ from $K$,
the {\em locus}
of $a$ over $C$,
denoted $\locus(a/C)$,
is the smallest Zariski closed subset of $\bbA_{C}^{m}$,
defined over $C$,
of which $a$ is an $K$-rational point.
Equivalently,
$\locus(a/C)$
is the affine variety defined by the vanishing
of the polynomials in the prime ideal
$I_{a}:=\{f\in C[x]\;|\;f(a)=0\}\unlhd C[x]$,
where $x$ is an $m$-tuple of variables.
\end{definition}

\begin{problem}\label{problem}
Let $a\in K^{m}$ and $b\in K^{n}$ be tuples.
Let $\pr_{x}:\bbA^{m+n}\rightarrow\bbA^{m}$ be the projection on to the first $m$-coordinates.
Our aim is to
locally describe the projection:
$\pr_{x}:\locus(a,b/C)(K)\rightarrow\pr_{x}\locus(a/C)(K)$.
That is, we want to describe the image of the map
\begin{align*}
	\pr_{x}:\locus(a,b/C)(K)\cap U\rightarrow\locus(a/C)(K)\cap\pr_{x}U,
\end{align*}
for some $\tau$-neighbourhood $U$ of $(a,b)$.
\end{problem}

The main tool we use is the Implicit Function Theorem for polynomials which, in the following form, is equivalent to henselianity of the topology (subject to being a $V$-topology).

\begin{fact}[{``Implicit Function Theorem'', cf \cite[(7.4) Theorem]{PZ78}}]\label{fact:IFT}
Let $(F,\tau)$ be a V-topological field.
Then $(F,\tau)$ is t-henselian if and only if,
for every $f\in K[X_{0},\ldots,X_{m-1},Y]$
and every $(a_{0},\ldots,a_{m-1},b)\in K^{m+1}$
such that $f(a_{0},\ldots,a_{m-1},b)=0$
and $D_{Y}f(a_{0},\ldots,a_{m-1},b)\neq0$,
there are $\tau$-open sets $U\subseteq K^{m}$
and $V\subseteq K$ such that
\begin{enumerate}[{\bf(i)}]
	\item $(a_{0},\ldots,a_{m-1},b)\in U\times V$ and
	\item $Z(f)\cap(U\times V)$ is the graph of a continuous function $U\rightarrow V$.
\end{enumerate}
\end{fact}

\begin{remark}
The topological-algebraic statements in this section,
including Facts~\ref{fact:IFT} and~\ref{fact:MIFT},
Lemma~\ref{lem:separable_projections},
and Proposition~\ref{prp:arbitrary_projections},
have analogues in enriched settings,
where for example
we work with a field equipped with analytic functions,
with respect to which it satisfies the appropriate Implicit Function Theorem.
Suitably adapted notions of ``henselianity'' and of ``locus'' are needed, as well as a more careful treatment of separability and inseparability.
We postpone discussion of this topic for future work.
\end{remark}

The following is a form of the multidimensional Implicit Function Theorem for polynomials,
as given in \cite[\S4]{Kuhlmann_book},
but restated for our context of fields equipped with henselian topologies.

\begin{fact}[{Cf \cite[\S4]{Kuhlmann_book}}]\label{fact:MIFT}
Recall our standing assumption \DAG.
Let $f_{1},\ldots,f_{n}\in K[X_{1},\ldots,X_{l}]$ with $n<l$.
Define the Jacobian
\begin{align*}
	\tilde{J}
&:=
\begin{pmatrix}
\frac{\partial f_{1}}{\partial X_{l-n+1}}
&
\ldots
&
\frac{\partial f_{1}}{\partial X_{l}}
\\
\vdots
&
\ddots
&
\vdots
\\
\frac{\partial f_{n}}{\partial X_{l-n+1}}
&
\ldots
&
\frac{\partial f_{n}}{\partial X_{l}}
\end{pmatrix},
\end{align*}
where $\tfrac{\partial}{\partial X_{i}}$ is the formal derivative with respect to $X_{i}$.
Assume that
$f_{1},\ldots,f_{n}$
admit a common zero
$a=(a_{1},\ldots,a_{l})\in K^{l}$
and that
$\det\tilde{J}(a)\neq0$.
Then there are some
$U_{1},\ldots,U_{l}\in\tau$
with $a\in\prod_{i=1}^{l}U_{i}$
such that for all
$(a_{1}',\ldots,a_{l-n}')\in\prod_{i=1}^{l-n}U_{i}$
there exists a unique
$(a_{l-n+1}',\ldots,a_{l}')\in\prod_{i=l-n+1}^{l}U_{i}$
such that
$(a_{1}',\ldots,a_{l}')$
is a common zero of
$f_{1},\ldots,f_{n}$.
\end{fact}
\begin{proof}
The analogous statement for the particular case 
of topologies induced by nontrivial henselian valuations is given in
\cite[\S4]{Kuhlmann_book}.
This version follows because $(K,\tau)$ is locally equivalent to a topological field $(L,\tau_{w})$, where $\tau_{w}$ is the topology induced by a nontrivial henselian valuation $w$ on $L$,
and the statement of this lemma is expressed by a local sentence
(in the sense of \cite[\S1]{PZ78}).
\end{proof}

For brevity, for Zariski constructible sets $A,B\subseteq\bbA_{C}^{n}$,
and a neighbourhood $U$ in a topology on $K^{n}$,
we write $A\localin_{U}B$ to mean that $A(K)\cap U\subseteq B(K)\cap U$,
and we write $A\localeq_{U}B$ to mean both $A\localin_{U}B$ and $B\localin_{U}A$.
The following lemma is just a basic fact of algebraic geometry, which we restate and prove in our language, for lack of a suitable reference.

\begin{lemma}\label{lem:TOOL}
Let $F/C$ be a field extension and let $a=(a_{j})_{j<n}\in F^{n}$.
For each $j<n$, choose $g_{j}\in C[X_{0},\ldots,X_{j}]$ and $h_{j}\in C[X_{0},\ldots,X_{j-1}]$
such that
if $a_{j}$ is algebraic over $C(a_{0},\ldots,a_{j-1})$ then
\begin{align*}
	\frac{g_{j}(a_{0},\ldots,a_{j-1},X_{j})}{h_{j}(a_{0},\ldots,a_{j-1})}
\end{align*}
is its minimal polynomial, otherwise
if $a_{j}$ is transcendental over $C(a_{0},\ldots,a_{j-1})$ then $g_{j}=0$ and $h_{j}=1$ are constant polynomials.
\begin{enumerate}[{\bf(i)}]
\item
For any finitely many polynomials $(f_{i})_{i<l}\subseteq C[X_{0},\ldots,X_{n-1}]$ with $f_{i}(a)=0$
there is a Zariski open set $U\subseteq F^{n}$, with $a\in U$, such that
\begin{align*}
	Z(g_{0},\ldots,g_{n-1})
&\localin_{U}
	Z(f_{0},\ldots,f_{l-1}).
\end{align*}
\item
There exists a Zariski open set $V\subseteq F^{n}$, with $a\in V$, such that for each $j\in\{0,\ldots,n\}$ we have
\begin{align*}
	\locus(a/C)
&\localeq_{V}
	\locus(a_{0},\ldots,a_{j-1}/C)\cap Z(g_{j},\ldots,g_{n-1}).
\end{align*}
In particular, taking $j=0$, we have
$\locus(a/C)\localeq_{V}Z(g_{0},\ldots,g_{n-1})$.
\end{enumerate}
\end{lemma}
\begin{proof}
To prove {\bf(i)}:
we proceed by induction on $n$.
The base case $n=0$ is vacuous.
As an inductive hypothesis, we suppose that the statement of the lemma holds for some $n<\omega$.
Let $a\in F^{n}$.
By Noetherianity, 
$I_{a}=\{f\in C[X]\mid f(a)=0\}$ is a finitely generated deal of $C[X]$.
Let $f'=(f'_{h})_{h<k}$ be a choice of generators,
so $Z(f')=\locus(a/C)$.
By the inductive hypothesis applied to these generators $f'$,
there is a Zariski-open set $U'\subseteq F^{n}$, with $a\in U'$, such that
$Z(g_{0},\ldots,g_{n-1})\localin_{U'}Z(f')$.
Since anyway $\locus(a/C)\subseteq Z(g_{0},\ldots,g_{n-1})$,
we have
$Z(g_{0},\ldots,g_{n-1})\localeq_{U'}Z(f')\localeq_{U'}\locus(a/C)$. 
Let $b\in F$ be any other element.
Let $f=(f_{i})_{i<l}\subseteq I_{a,b}$ be finitely many polynomials such that $f_{i}(a,b)=0$.
Let $i<l$.
Case {\bf(a)}:
suppose first that $f_{i}(a,Y)$ is the zero polynomial in $Y$.
Then $f_{i}(X,Y)=f_{i}(X)\in I_{a}$.
Case {\bf(b)} is the more interesting case:
suppose that $f_{i}(a,Y)$ is not the zero polynomial in $Y$.
This implies that $b$ is algebraic over $C(a)$, and in particular that $g_{n}(a,Y)$ is nonzero.
Since $g_{n}(a,Y)/h_{n}(a)$ is the minimal polynomial of $b$ over $C(a)$,
there exist $p_{i}\in C[X,Y]$ and $q_{i}\in C[X]$ such that
$q_{i}(a)\neq0$
and
\begin{align*}
	\frac{p_{i}(a,Y)g_{n}(a,Y)}{q_{i}(a)h_{n}(a)}&=f_{i}(a,Y).
\end{align*}
This means that there exists $r_{i}\in C[X]$ such that
$r_{i}(a)=0$
and
$p_{i}g_{n}=q_{i}h_{n}f_{i}+r_{i}$ in $C[X,Y]$.
Let
\begin{align*}
	U=U'\setminus Z(h_{n},q_{i}\mid\text{$i<l$ is in case {\bf(b)}})(F),
\end{align*}
be the complement in $U'$ of the vanishing locus of these denominator polynomials,
as a Zariski open subset of $F^{n+1}$ with $(a,b)\in U$.
It remains to verify that
$Z(g_{0},\ldots,g_{n})\localin_{U}Z(f)$.
Let $(a',b')\in Z(g_{0},\ldots,g_{n})(F)\cap U$,
and $i<l$.
This implies that $a'\in\locus(a/C)(F)$.
If $i$ is in case {\bf(a)} then $f_{i}\in I_{a}$, so already $f_{i}(a',b')=0$.
Otherwise if $i$ is in case {\bf(b)} we have
$r_{i}(a')=0$,
$q_{i}(a')\neq0$,
and $h_{n}(a')\neq0$.
Then it follows from $g_{n}(a',b')=0$ that also in this case we have $f_{i}(a',b')=0$.
Thus $(a',b')\in Z(f_{0},\ldots,f_{l-1})(F)$,
which that proves
$Z(g_{0},\ldots,g_{n})\localin_{U}Z(f)$,
finishing the induction.

To prove {\bf(ii)}:
let $j\in\{0,\ldots,n\}$
and notice that $\locus(a_{0},\ldots,a_{j-1}/C)\subseteq Z(g_{0},\ldots,g_{j-1})$.
Applying {\bf(i)} to a set of generators $f$ for $I_{(a_{0},\ldots,a_{j-1})}$, we have
$Z(g_{0},\ldots,g_{j-1})\localeq_{V_{j}}\locus(a_{0},\ldots,a_{j-1}/C)$
for some Zariski open $V_{j}\subseteq F^{j}$.
Finally let
$V:=\bigcap_{j=0}^{n}V_{j}\times F^{n-j}$.
It is easy to see that this set $V$ is a Zariski open subset of $F^{n}$ with $a\in V$,
and that satisfies the statement of the lemma.
\end{proof}

The next lemma improves on \cite[Lemma 19]{A19} by allowing arbitrary separable extensions $C(a,b)/C(a)$,
instead of assuming that $a$ is a separating transcendence basis of $C(a,b)/C$.

\begin{lemma}[{Separable projection of loci}]\label{lem:separable_projections}
Recall our standing assumption \DAG.
Let $(a,b)\in K^{m+n}$
and suppose that $C(a,b)/C(a)$ is separable.
Let $b_{1}\subseteq b$ a separating transcendence base of $C(a,b)/C(a)$
[which exists by Lemma~\ref{lem:Mac_Lane_I}],
let $b_{2}:=b\setminus b_{1}$,
and let $n=n_{1}+n_{2}$ be the corresponding partition of $n$.
There exists
$\tau$-neighbourhoods
$U$, $V_{1}$, and $V_{2}$
of $a$, $b_{1}$, and $b_{2}$,
respectively,
such that
\begin{align*}
	\locus(a,b/C)(K)\cap(U\times V_{1}\times V_{2})
\end{align*}
is the graph of a continuous function
\begin{align*}
	f:(\locus(a/C)(K)\times K^{n_{1}})\cap(U\times V_{1})&\rightarrow V_{2}.
\end{align*}
In particular, the projection of
$\locus(a,b/C)(K)$
onto the first $m$ coordinates contains 
$\locus(a/C)(K)\cap U$.
\end{lemma}
We say that the projection
$\locus(a,b/C)(K)\rightarrow\locus(a/C)(K)$
is ``{locally surjective}''.
\begin{proof}
We adapt the proof of \cite[Lemma 19]{A19},
though in the context of henselian topologies as in Fact~\ref{fact:MIFT}.
Let $X=(X_{i})_{i<m}$ and $Y=(Y_{j})_{j<n}$ be tuples of variables, of lengths $m$ and $n$, to correspond to $a$ and $b$, respectively.
We reorder $b$ and $Y$, if necessary, so that the variables $(Y_{0},\ldots,Y_{n_{1}-1})$ correspond to $b_{1}$ and whereas $(Y_{n_{1}},\ldots,Y_{n-1})$ correspond to $b_{2}$.
For each $j\in\{n_{1},\ldots,n-1\}$,
let
$g_{j}\in C[X,Y_{j_{0}}\mid j_{0}\leq j]$
and
$h_{j}\in C[X,Y_{j_{0}}\mid j_{0}<j]$
be such that
$h_{j}(a,b_{0},\ldots,b_{j-1})\neq0$
and
\begin{align*}
	\frac{g_{j}(a,b_{0},\ldots,b_{j-1},Y_{j})}{h_{j}(a,b_{0},\ldots,b_{j-1})}
\end{align*}
is the minimal polynomial of $b_{j}$ over $C(a,b_{j_{0}}\mid j_{0}<j)$.
For $j<n_{1}$, we let $g_{j}$ be the zero polynomial, and we set $h_{j}=1$ (or even leave it undefined).
We are in exactly the setting of Lemma~\ref{lem:TOOL}~{\bf(ii)},
so by that lemma
there is a Zariski open set $W\subseteq K^{m+n}$
such that
\begin{align*}
	\locus(a,b/C)
\localeq_{W}
	\locus(a/C)\times Z(g_{0},\ldots,g_{n-1})
\end{align*}
Clearly $Z(g_{0},\ldots,g_{n-1})=Z(g_{n_{1}},\ldots,g_{n-1})$.
Thus
\begin{align*}
	\locus(a,b/C)
\localeq_{W}
	(\locus(a/C)\times\bbA_{C}^{n})\cap Z(g_{n_{1}},\ldots,g_{n-1}).
\end{align*}
For each $j\in\{n_{1},\ldots,n-1\}$,
since $b_{j}$ is separably algebraic over $C(a,b_{0},\ldots,b_{j-1})$,
we have
both
$g_{j}(a,b_{0},\ldots,b_{j})=0$
and
$\tfrac{\partial}{\partial Y_{j}}g_{j}(a,b_{0},\ldots,b_{j})\neq0$.
Again we define
the Jacobian
\begin{align*}
	\tilde{J}
&:=
\begin{pmatrix}
\frac{\partial g_{n_{1}}}{\partial Y_{n_{1}}}
&
\ldots
&
\frac{\partial g_{n_{1}}}{\partial Y_{n-1}}
\\
\vdots
&
\ddots
&
\vdots
\\
\frac{\partial g_{n-1}}{\partial Y_{n_{1}}}
&
\ldots
&
\frac{\partial g_{n-1}}{\partial Y_{n-1}}
\end{pmatrix}
\end{align*}
and observe that $\tilde{J}(a,b)$ is a lower triangular matrix with no zeroes on the diagonal, thus $\det\tilde{J}(a,b)\neq0$.
By the Implicit Function Theorem for polynomials, Fact~\ref{fact:MIFT},
there are
$\tau$-neighbourhoods
$U\subseteq K^{m}$ of $a$,
$V_{1}\subseteq K^{n_{1}}$ of $b_{1}$,
and
$V_{2}\subseteq K^{n_{2}}$ of $b_{2}$,
such that
\begin{align*}
	Z(g_{n_{1}},\ldots,g_{n})(K)
	\cap(U\times V_{1}\times V_{2})
\end{align*}
is the graph of a continuous function
$U\times V_{1}\rightarrow V_{2}$.
In fact, by continuity,
and since $\tau$ refines the Zariski topology (on $K$-rational points of affine space),
we may even ensure that
$U\times V_{1}\times V_{2}\subseteq W$.
Simply intersecting with
the $K$-rational points of
$\locus(a/C)\times\bbA_{C}^{n}$,
we have that
\begin{align*}
	\locus(a,b/C)(K)
	\cap(U\times V_{1}\times V_{2})
\end{align*}
is the graph of a continuous function
$(\locus(a/C)(K)\times K^{n_{1}})\cap(U\times V_{1})\rightarrow V_{2}$.
\end{proof}

In Proposition~\ref{prp:arbitrary_projections}, to address Problem~\ref{problem}, we apply our assembled ingredients to describe projections of loci having removed the assumption that $C(a,b)/C(a)$ is separable.
Under our standing assumption \DAG,
and following Remark~\ref{rem:sigma},
we denote by
$\sigma_{K/b/c}:\bbA_{C}^{\ell}\rightarrow\bbA_{C}^{m}$
the coordinate projection
that maps
$\lambda_{K/b/c}a\mapsto a$,
where
$c\in\pB[C]$,
$(a,b)\in K^{m+n}$,
and $\ell=|\lambda_{K/b/c}a|$.

\begin{proposition}[{Arbitrary projection of loci}]\label{prp:arbitrary_projections}
Recall our standing assumption \DAG.
Let $c\in\pB[C]$
and let $(a,b)\in K^{m+n}$.
Let $b_{1}\subseteq b$ a separating transcendence base of $C(\lambda_{K/b/c}a,b)/C(\lambda_{K/b/c}a)$
[which exists by Lemma~\ref{lem:Mac_Lane_I}],
let $b_{2}:=b\setminus b_{1}$,
and let $n=n_{1}+n_{2}$ be the corresponding partition of $n$.
There exist $\tau$-neighbourhoods $U$, $V_{1}$, and $V_{2}$ of $\lambda_{K/b/c}a$, $b_{1}$, and $b_{2}$, respectively
such that
\begin{enumerate}[{\bf(i)}]
\item
$\locus(a,b/C)(K)$
contains the image 
of the graph of a continuous function
\begin{align*}
	f:(\locus(\lambda_{K/b/c}a/C)(K)\times K^{n_{1}})\cap(U\times V_{1})&\rightarrow V_{2},
\end{align*}
under the coordinate projection
$\sigma_{K/b/c}\times\id^{n}:\bbA_{C}^{\ell+n}\rightarrow\bbA_{C}^{m+n}$,
and
\item
the image of the projection
$\pr_{m}:\locus(a,b/C)(K)\rightarrow\locus(a/C)(K)$
onto the first $m$ coordinates contains 
the image of
$\locus(\lambda_{K/b/c}a/C)(K)\cap U$
under the coordinate projection
$\sigma_{K/b/c}:\bbA_{C}^{\ell}\rightarrow\bbA_{C}^{m}$.
\end{enumerate}
\end{proposition}
\begin{figure}
\begin{align*}
  	\xymatrix@=4.0em{
	\locus(a,b/C)
        \ar@{->}@/^0.00pc/[d]_{\pr_{m}}
&
	\locus(\lambda_{K/b/c}a,b/C)
        \ar@{->}@/^0.00pc/[d]_{\pr_{\ell}}
        \ar@{->}@/^0.00pc/[l]^{\sigma_{K/b/c}\times\id^{n}}
\\
 	\locus(a/C)
&
 	\locus(\lambda_{K/b/c}a/C)
        \ar@{->}@/^0.00pc/[l]^{\sigma_{K/b/c}}
  	}
\end{align*}
\caption{Illustration of Proposition~\ref{prp:arbitrary_projections}}
\label{fig:1}
\end{figure}
\begin{proof}
Forgetting the matter of the $K$-rational points for a moment,
the projection
$\pr_{m}:\locus(a,b/C)\rightarrow\locus(a/C)$
(restricted from the coordinate projection $\bbA_{C}^{m+n}\rightarrow\bbA_{C}^{m}$)
is associated to the extension
of function fields
$C(a,b)/C(a)$.
Observe that if this extension is separable,
we may directly apply
Lemma~\ref{lem:separable_projections},
which gives the statement of the proposition, since in this case $\lambda_{K/b/c}a=a$.
In general,
at least we have that the extension
$C(\lambda_{K/b/c}a,b)/C(\lambda_{K/b/c}a)$
is always separable,
by
Lemma~\ref{lem:finite}
and
Definition~\ref{def:finite}.
This extension is associated to
the projection
$$\pr_{\ell}:\locus(\lambda_{K/b/c}a,b/C)\rightarrow\locus(\lambda_{K/b/c}a/C),$$
which is restricted from $\bbA_{C}^{\ell+n}\rightarrow\bbA_{C}^{\ell}$.
This projection, together with
the projection
$$\sigma_{K/b/c}:\locus(\lambda_{K/b/c}a/C)\rightarrow\locus(a/C),$$
itself restricted from $\bbA_{C}^{\ell}\rightarrow\bbA_{C}^{m}$,
naturally forms the commutative square
illustrated in Figure~\ref{fig:1}.
Of course the same diagram makes sense and commutes when we restrict our attention to the $K$-rational points.
Applying 
Lemma~\ref{lem:separable_projections}
to the right-hand side of the square,
we obtain the $\tau$-neighbourhoods
$U$, $V_{1}$, and $V_{2}$
such that
\begin{align*}
	\locus(\lambda_{K/b/c}a,b/C)(K)\cap(U\times V_{1}\times V_{2})
\end{align*}
is the graph of a continuous function
\begin{align*}
	f:(\locus(\lambda_{K/b/c}a/C)(K)\times K^{n_{1}})\cap(U\times V_{1})&\rightarrow V_{2}.
\end{align*}
Claim {\bf(i)}
follows by applying $\sigma_{K/b/c}\times\id^{n}$.
Next we observe that the image of 
$$\pr_{m}:\locus(a,b/C)(K)\rightarrow\locus(a/C)(K)$$
must contain the image of
$\locus(\lambda_{K/b/c}a,b/C)(K)$
under the composition
$\pr_{m}\circ(\sigma_{K/b/c}\times\id^{n})$.
Since the square of maps commutes,
$\pr_{m}(\locus(a,b/C)(K))$
contains the image of
$\locus(\lambda_{K/b/c}a/C)(K)$
under the composition
$\sigma_{K/b/c}\times\pr_{\ell}$.
By ``local surjectivity'',
the latter contains the image of
$\locus(\lambda_{K/b/c}a/C)\cap U$
under $\sigma_{K/b/c}$,
which proves {\bf(ii)}.
\end{proof}

\begin{proof}[{Proof of Theorem~\ref{thm:intro_3}}]
Let $X\subseteq K^{m}$ be a existentially $\Lring(C)$-definable,
as in the statement of the theorem.
By standard reductions in the first-order theory of fields,
$X$ is the projection onto $K^{m}$ of the set of $K$-rational points of an affine subvariety $V\subseteq\bbA_{C}^{m+n}$,
i.e.~$\pr_{m}V(K)=X$.
For $a\in X$, there exists $b\in K^{n}$ such that $(a,b)\in V(K)$,
and therefore
$\locus(a,b/C)\subseteq V$
and
$\pr_{m}\locus(a,b/C)(K)\subseteq\pr_{m}V(K)=X$.
By Proposition~\ref{prp:arbitrary_projections}~{\bf(ii)},
there exists a $\tau$-neighbourhood $U$
such that
$\pr_{m}\locus(a,b/C)(K)$
contains the image of
$\locus(\lambda_{K/b/c}a/C)(K)\cap U$
under the coordinate projection
$\sigma_{K/b/c}$.
\end{proof}

\begin{remark}\label{rem:strengths_and_shortcomings}
We briefly comment informally on the strengths and shortcomings of Proposition~\ref{prp:arbitrary_projections}.
It is clear that this result is not anything like as powerful as a true quantifier elimination result.
Indeed, quantifier elimination cannot possibly hold at this generality,
i.e.~in the theory of {\em all} henselian nontrivially valued fields of a fixed characteristic $p$.
Even relative to theories of value group and residue field, we do not have a complete theory,
for example by the counterexample developed in \cite{Kuh01}.
Nevertheless, the above Proposition shows that,
at least locally in the henselian topology,
around a sufficiently generic point
(since $(a,b)$ is a generic point of $\locus(a,b/C)$),
the image of a coordinate projection
on the $K$-rational points
is exactly the 
$K$-rational points of
a set defined by a conjunction of atomic $\Lringlambda$-formulas from the type of $a$ over $C$.
This may be seen as giving some kind of normal form for subsets of henselian valued fields that are defined by existential $\Lring$-formulas,
at least locally around a given point.
\end{remark}

In the present article, the principal application of Proposition~\ref{prp:arbitrary_projections} is the following theorem.

\begin{theorem}\label{thm:main_ish}
Let $K$ be a field equipped with a henselian topology $\tau$,
let $C\subseteq K$ be a subfield,
and
let $A\subseteq K$ be a subset defined by an existential $\Lring$-formula with parameters from $C$.
Exactly one of {\bf(i)} or {\bf(ii)} holds:
	\begin{enumerate}[{\bf(i)}]
		\item
		$A$ is a finite subset of the relative algebraic closure of $\Lambda_{K}C$ in $K$,
		\item
		$A$ is infinite and there is a definable injection 
		$g:U_{1}^{\circ}\rightarrow A$,
		perhaps involving extra parameters,
		where $U_{1}^{\circ}$ is a non-empty Zariski open subset of a $\tau$-neighbourhood $U_{1}$,
		such that each element of $U_{1}^{\circ}$ is interalgebraic with its image under $g$
		over the parameters.
	\end{enumerate}
	Moreover, if $C\subseteq K^{(p^{\infty})}$,
	then
	exactly one of
	{\bf(i)} and {\bf(ii')}
	holds
	where:
	\begin{enumerate}[{\bf(i')}] 
	\setcounter{enumi}{1}
		\item
		there exists $m<\omega$ such that $A$ contains the $K^{(p^{m})}$-points of a nonempty $\tau$-open set.
	\end{enumerate}
\end{theorem}
\begin{proof}
By our hypotheses and the usual reductions in the first-order theory of fields,
there is a positive quantifier-free $\Lring$-formula $\varphi(x,y,z)$ 
and a $z$-tuple $c_{0}\subseteq C$ of parameters such that
the formula
$\exists y\;\varphi(x,y,c_{0})$
defines the set $A$ in $K$.
Denote $D:=\Lambda_{K}C$.
As a special case, we first suppose that 
there exists $a\in A$ that is transcendental over $D$.
There exists $b\in K^{n}$
such that
$K\models\varphi(a,b,c_{0})$.
It follows that
$\pr_{1}(\locus(a,b/D)(K))\subseteq A$.
Let $c\in\pB[D]$
and
observe that $\lambda_{K/b/c}a$ is a finite tuple,
by Theorem~\ref{thm:intro_1}~{\bf(iii)}.
Decompose
$\lambda_{K/b/c}a=\concat{{e_{1}}}{{e_{2}}}$
by choosing $e_{1}$ to be a separating transcedence basis of $D(\lambda_{K/b/c}a)/D$,
by Lemma~\ref{lem:Mac_Lane_I}~{\bf(iv)}.
By Lemma~\ref{lem:separable_projections},
there are $\tau$-neighbourhoods $U,V$
such that
\begin{align*}
	\locus(\lambda_{K/b/c}a/D)\cap(U\times V)
\end{align*}
is the graph of a continuous function
\begin{align*}
	f:\locus(e_{1}/D)\cap U\rightarrow V.
\end{align*}
We now repeat an argument used in
\cite[Lemma 22]{A19}:
Since $e_{1}$ is algebraically independent over $C$, and $a$ is algebraically dependent on $D(e_{1})$ but transcendental over $D$,
there exists a singleton $e_{1,1}\in e_{1}$
such that
$e_{1,1}$ and $a$ are interalgebraic over $E:=D(e_{1,2})$,
where
$e_{1,2}:=e_{1}\setminus\{e_{1,1}\}$.
By reordering $e_{1}$ if necessary
we may even suppose
$e_{1}=\concat{{e_{1,1}}}{{e_{1,2}}}$.
Let $N\in\mathbb{N}$ be such that the interalgebraicity of $e_{1,1}$ and $a$ over $E$ is witnessed by a polynomial $h\in E[X,Y]$ of degree at most $N$
with coefficients consisting of $\Lringlambda$-terms in $c_{0}\cup e_{1,2}$
of complexity at most $N$ (for any reasonable notion of the complexity of terms).
Let $U_{1}$ be a $\tau$-neighbourhood of $e_{1,1}$
and $U_{2}$ be a $\tau$-neighbourhood of $e_{1,2}$,
chosen such that $U_{1}\times U_{2}\subseteq U$.
Then $f$ restricts to a continuous function
\begin{align*}
	\locus(e_{1,1},e_{1,2}/D)\cap(U_{1}\times U_{2})\rightarrow V.
\end{align*}
It follows that for the continuous map
\begin{align*}
	g:U_{1}&\rightarrow\pr_{1}(\locus(a,b/E))\\
	u&\mapsto\sigma_{K/b/c}(u,e_{1,2},f(u,e_{1,2})),
\end{align*}
whenever $u\in U_{1}$ is transcendental over $E$,
$E(u,g(u))$ is isomorphic to $E(e_{1,1},a)$,
over $E$, via $(u,g(u))\mapsto(e_{1,1},a)$.
In particular $u$ and $g(u)\in\pr_{x}(\locus(a,b/C))\subseteq A$
are interalgebraic over $E$, and this is also witnessed by the polynomial $h$.
When restricted to the subset of $U_{1}$ consisting of those $u$ trancendental over $E$, $g$ is even a bijection.
By compactness both the bijectivity and the interalgebraicity hold for a Zariski-open subset $U_{1}^{\circ}$ of $U_{1}$.
In particular, since $K$ is infinite, so is the image of $g$.
This proves that {\bf(ii)} holds in this special case.

Next suppose simply that $A$ is infinite.
Then passing to an $\aleph_{0}$-saturated extension $K^{*}\succeq K$, the set $A^{*}$ defined in $K^{*}$ by $\exists y\;\varphi(x,y,c_{0})$
contains an element $a$ that is transcendental over $D$.
We apply the argument of the previous paragraph to $A^{*}$ in $K^{*}$,
and observe that the conclusion is elementary, thus also holds for $A$ in $K$.
This proves that {\bf(ii)} holds.
If $A$ is both finite and contains no elements transcendental over $D$, then {\bf(i)} certainly holds.
This proves the dichotomy.

Finally, we suppose $C\subseteq K^{(p^{\infty})}$,
whence $D=C^{\perf}\subseteq K^{(p^{\infty})}\subseteq(K^{*})^{(p^{\infty})}$.
Observe that
$\Lambda_{K^{*}}D(a)=D(a^{p^{-m_{1}}})$,
where $m_{1}$ is the unique natural number such that $a\in(K^{*})^{(p^{m_{1}})}\setminus(K^{*})^{(p^{m_{1}+1})}$,
or is infinity if $a\in K^{(p^{\infty})}$.
It follows that there is a (possibly different) natural number $m$ such that $b$ is separable over
$D(a^{p^{-m}})$.
By Lemma~\ref{lem:separable_projections},
there exists a $\tau$-neighbourhood $U_{1}$ of $a^{p^{-m}}$ such that
$\locus(a^{p^{-m}}/D)(K)\cap U_{1}\subseteq\pr_{1}(\locus(a^{p^{-m}},b/D))$.
Since we have supposed $a$ to be transcendental over $D$,
$\locus(a^{p^{-m}}/D)(K)=K$,
and therefore it follows that
$U_{1}^{(p^{m})}\subseteq\pr_{1}(\locus(a,b/D)(K))\subseteq X$.
Finally we note that $a\in U_{1}^{(p^{m})}$
and there is a $\tau$-neighbourhood $U_{1}'$ of $a$ such that
$U_{1}'\cap K^{(p^{m})}\subseteq U_{1}^{(p^{m})}$,
which proves that {\bf(ii')} holds.
\end{proof}

We denote the
existential $\Lring$-algebraic closure
(in the model-theoretic sense)
of $A\subseteq K$
by $\acl_{\exists}^{K}A$.
This is the union of those finite subsets of $K$ that are definable by an existential $\Lring$-formula,
with parameters from $A$.
Similarly, we denote the
existential $\Lring$-definable closure
(again in the model-theoretic sense)
by $\dcl_{\exists}^{K}A$.

\begin{corollary}\label{cor:intro_bII}
Let $K$ be a field with a henselian topology $\tau$,
and let $B\subseteq K$ be a subset.
Then
\begin{enumerate}[{\bf(i)}]
\item
$\dcl_{\exists}^{K}(B)$
contains $\rmLambda_{K}\bbF(B)$
and is contained in 
the (field theoretic) relative algebraic closure of $\rmLambda_{K}\bbF(B)$ in $K$,
and
\item
$\acl^{K}_{\exists}(B)$
is equal to 
the (field theoretic) relative algebraic closure of $\rmLambda_{K}\bbF(B)$ in $K$.
\end{enumerate}
\end{corollary}
\begin{proof}
In any field $F$, with $B\subseteq F$, we have $\rmLambda_{F}\bbF(B)\subseteq\dcl_{\exists}^{F}(B)$,
for example by Theorem~\ref{thm:intro_1}, so in particular this holds for $F=K$.
Now, let $A\subseteq K$ be a subset defined by an existential $\Lring$-formula with parameters from $C:=\Lambda_{K}\bbF(B)$.
If any element of $A$ is transcendental over $C$ then $A$ is infinite, by Theorem~\ref{thm:main_ish}.
In particular, $\acl_{\exists}^{K}(B)$ is a subset of the relative algebraic closure of $C$ in $K$,
proving {\bf(i)}.
Conversely, every element $a$ of the relative algebraic closure of $C$ in $K$ is contained in a finite set $A$ that is definable over $C$
by an existential $\Lring$-formula $\varphi(x,c)$, with parameters $c$ from $C$.
Since $C$ itself is a subset of $\dcl_{\exists}^{K}(B)$,
as we have already established,
both $\varphi$ and $c$ may be replaced by another existential $\Lring$-formula $\psi(x,b)$ with parameters $b$ from $B$,
such that $A$ is defined by $\psi(x,b)$ in $K$.
This proves {\bf(ii)}.
\end{proof}

Corollary~\ref{cor:intro_b} follows immediately.

\section{Separably tame valued fields}
\label{section:STVF}

The main aim of this section is to
extend the account of separably tame valued fields,
as developed by Kuhlmann and Pal
(\cite{KuhlmannPal})
to allow infinite imperfection degree.
For this we will make (rather mild) use of 
Lambda closure $\rmLambda_{F}$
from section~\ref{section:Lambda}.
Let
$\Lring=\{+,\times,-,0,1\}$
be the language of rings,
let $\Loag$ be the language of ordered abelian groups,
and let
$\Lval$
be the three-sorted language of valued fields
with sorts $\mathbf{K}$, $\mathbf{k}$, and $\bfGamma$.
The first two are endowed with $\Lring$, and the last with $\Loag$,
moreover there is a symbol for the valuation map from $\mathbf{K}$ to $\bfGamma$,
and for the residue map from $\mathbf{K}$ to $\mathbf{k}$.

\begin{remark}\label{rem:notation}
In this section we write $K,K_{i},\ldots,$ etc.,
for expansions of valued fields.
The valuation will be usually be denoted by $v$,
with subscripts or other decorations used to indicate to which valued field the valuation belongs,
e.g.~$v_{i}$ is the valuation from $K_{i}$.
Likewise $\Gamma_{i}=v_{i}K_{i}$ and $k_{i}=K_{i}v_{i}$ are the value group and residue field, respectively, of $K_{i}$.
\end{remark}

In the present section,
we will follow a convention that differs from that of section~\ref{section:Lambda},
in which $p$ played the role of the characteristic exponent.
From now on, we are concerned with valued fields $K$ of equal characteristic,
and $p$ will always represent the characteristic (of both $K$ and of its residue field $k$), and $\hat{p}$ the corresponding characteristic exponent:

\begin{convention}\label{convention:characteristic_exponent}
For $p\in\mathbb{P}\cup\{0\}$, we let $\hat{p}=p$ if $p\in\mathbb{P}$, and $\hat{p}=1$ if $p=0$.
\end{convention}

\begin{remark}\label{rem:alternative_languages}
When formalizing valued fields in model theory,
we have the usual choice of alternative languages.
Instead of $\Lval$,
we might use
a one-sorted language
$\Lval^{1}:=\Lring\cup\{O\}$, where $O$ is a unary predicate symbol, intended to be interpreted by the valuation ring;
or we might use a two-sorted language $\Lval^{2}$ with sorts $\mathbf{K}$ and $\bfGamma$, and with a function symbol from $\mathbf{K}$ to $\bfGamma$.
For the results of this paper, the precise choice of language of valued fields does not matter.
For example, both Theorem~\ref{thm:intro_2} and Corollary~\ref{cor:intro_a}
remain true when replacing $\Lval$ by another language $\Lang$,
provided that the $\Lval$- and $\Lang$-structures are biinterpretable,
and that the interpretations of both the value group and residue field are by both existential and universal formulas.
This latter condition ensures that, for example, existential $\Lring$-sentences in the theory of the residue field are interpreted by existential sentences in the $\Lang$-theory of the valued field.
In particular, these conditions hold for the languages $\Lval^{1}$ and $\Lval^{2}$.
\end{remark}

We denote by
$\Lvlambda=\Lval\cup\Llambda$
the expansion of $\Lval$ by symbols for the parameterized lambda functions,
uniform across all characteristics,
as introduced in~\ref{section:lambda_language}.
For an expansion $\Lang\supseteq\Lval$,
any $\Lang$-theory $T$ of valued fields is in particular an expansion of an $\Lring$-theory of fields,
thus $T_{\lambda}$ denotes its natural $(\Lang\cup\Llambda)$-expansion,
as described in Definition~\ref{def:theories}.

\begin{fact}\label{fact:separable_maps_III}
The analogue of Facts~\ref{fact:separable_maps} and~\ref{fact:separable_maps_II} applies to embeddings of valued fields:
that is,
each valued field $F$ admits a natural expansion $\tilde{F}\models\Th(F)_{\lambda}$ to an $\Lvlambda$-structure,
and an $\Lval$-embedding between $F_{1}$ and $F_{2}$
is in fact
an $\Lvlambda$-embedding between $\tilde{F}_{1}$ and $\tilde{F}_{2}$ 
if and only if it is separable, as an embedding of fields.
\end{fact}

\begin{definition}\label{def:STVF}
A valued field
$K$
is {\em separably tame}
if it is separably defectless,
has perfect residue field,
and $\hat{p}$-divisible value group.
Let
$\STVF$
be the $\Lval$-theory
of separably tame valued fields.
For $p\in\mathbb{P}\cup\{0\}$,
we let
$\STVF_{p}:=\STVF\cup\Xth_{p}$
be the theory of separably tame valued fields of equal characteristic $p$.
For $(p,\frakI)\in\mathbb{P}\times(\mathbb{N}\cup\{\infty\})$,
we let
$\STVF_{p,\frakI}:=\STVF_{p}\cup\Xth_{p,\frakI}$
be the theory $\STVF_{p}$ extended by axioms
for the elementary imperfection degree to be $\frakI$.
To any of these theories
the superscript "$^{\eq}$" will indicate the addition of axioms to ensure that the valued field is of equal characteristic,
though of course in the case of positive characteristic, which is our main fare, equal characteristic is automatic.
\end{definition}

For example,
$\STVF_{0}^{\eq}$
is the $\Lval$-theory of separably tame valued fields of equal characteristic zero
(which are in fact automatically tame).
We recall the following theorem.

\begin{theorem}[{\cite[Theorem 1.2]{KuhlmannPal}}]
\theoremintroKP
\end{theorem}

First, a small detail:
we prefer to write $\AKE^{\preceq_{\exists}}$ where Kuhlmann and Kuhlmann--Pal write $\AKE^{\exists}$ because we wish to include principles like $\AKE^{\equiv_{\exists}}$, and the earlier notation risks ambiguity.

We extend this theorem by strengthening the underlying embedding lemma
from \cite{KuhlmannPal},
which is closely based on the one from \cite{Kuh16}.
The AKE principles may then be stated uniformly for the class of separably tame valued fields of equal characteristic.
In particular, this extends the known Ax--Kochen/Ershov phenomena to the case of infinite imperfection degree.

\begin{remark}\label{rem:SCVF}
Let
$\SCVF$
be the theory of
separably closed
valued fields,
in the language
$\Lval$
of valued fields.
It is known since work of Delon
(e.g.~\cite{Delon82})
that the completions of $\SCVF$ are
$\SCVF_{0}$
and
$\SCVF_{p,\frakI}$,
for $(p,\frakI)\in\mathbb{P}\times(\mathbb{N}\cup\{\infty\})$.
Indeed, 
Hong showed in \cite{Hong}
that
$\SCVF$ has QE in $\Lvlambda$, for $p>0$.
\end{remark}

The following two theorems, due to Kuhlmann and Knaf--Kuhlmann, are the most powerful ingredients of the Embedding Lemma, in all of its forms:
those from \cite{Kuh16,KuhlmannPal} and Theorem~\ref{thm:EP}.

\begin{theorem}[{Strong Inertial Generation, {\cite[Theorem 3.4]{KK}, \cite[Theorem 1.9]{Kuh16}}}]\label{thm:SIG}
Let $L/K$ be a function field without transcendence defect, where $K$ is defectless.
Suppose also that $Lv/Kv$ is separable and $vL/vK$ is torsion free.
Then $L/K$ is strongly inertially generated.
\end{theorem}

\begin{theorem}[{Henselian Rationality, {\cite[Theorem 1.10]{Kuh16}}}]\label{thm:HR}
Let $L/K$ be an immediate function field (a finitely generated and regular extension) of dimension $1$,
where $K$ is separably tame.
Then $L\subseteq K(b)^{h}$ for some $b\in L^{h}$.
\end{theorem}

\subsection{The Lambda relative embedding property of separably tame valued fields}

Let $\Lang$ be an expansion of $\Lval$,
and let $\mathbf{C}$ be a class of $\Lang$-structures
expanding valued fields.

\begin{definition}[Lambda relative embedding property]\label{def:LREP}
We say that $\mathbf{C}$ has the
{\em Lambda relative embedding property}
(\EP)
if
\begin{enumerate}[---]
\item
for all $K_{1},K_{2}\in\mathbf{C}$
that
extend
a
separably tame
$K$
for which
\begin{enumerate}[{\bf(i)}]
\item
$K_{1}/K$ and \ul{$K_{2}/K$} are separable,
\item
\ul{$\impdeg(K_{1}/K)\leq\impdeg(K_{2}/K)$},
\item
$K_{1}$ is $\aleph_{0}$-saturated, $K_{2}$ is $|K_{1}|^{+}$-saturated,
\item
$vK_{1}/vK$ is torsion free and $K_{1}v/Kv$ is separable,
and
\item
$\rho:vK_{1}\underset{vK}{\rightarrow}vK_{2}$
and
$\sigma:K_{1}v\underset{Kv}{\rightarrow}K_{2}v$;
\end{enumerate}
\item
there exists
a \ul{separable} embedding
$\iota:K_{1}\rightarrow K_{2}$
inducing $\rho$ and $\sigma$.
\end{enumerate}
\end{definition}

\begin{remark}
We compare the \EP\ with the SREP, as expressed in \cite[\S 4]{KuhlmannPal}.
The points at which \EP\ differs from SREP are underlined, above,
with the key strengthened conclusion of \EP\ also underlined.
Strictly speaking, the two properties are incomparable:
the hypotheses are stronger,
i.e.~the extension $K_{2}/K$ is separable
and we suppose an inequality between imperfection degrees,
but the conclusion of the \EP\
is also stronger,
i.e.~the embedding $\iota$ is separable.
\end{remark}

\begin{remark}
The hypothesis {\bf(ii)} on imperfection degrees is a natural one,
given that our aim is to separably embed $K_{1}$ into $K_{2}$ over $K$.
Regarding {\bf(iv)}, note that both \EP\ and SREP suppose $K_{1}v/Kv$ to be separable,
but this is redundant in the case that $v$ is nontrivial on $K$, because then $Kv$ is perfect, since $K$ is separably tame.
Similarly, note that $\sigma$ is not assumed to be separable in {\bf(v)},
however this is automatic when $v$ is nontrivial on $K$, for then again $Kv$ is perfect.
\end{remark}

\begin{remark}
The REP,
as expressed in \cite{Kuh16},
appears to have weaker hypotheses
than SREP and \EP\
(aside from the obvious issues around separability),
but this is not a material distinction.
The conjunction of the hypothesis that $K$ is defectless
with the shared hypothesis {\bf(iv)}
is essentially equivalent to our hypothesis that $K$ is separably tame:
whenever REP is to be verified in a class of (separably) tame valued fields,
any common valued subfield $K$ satisfying the hypotheses is necessarily (separably) tame.
\end{remark}

\begin{lemma}[Separable going down, {\cite[Lemma 2.17]{KuhlmannPal}/\cite[Lemma 3.15]{Kuh16}}]\label{lem:going_down2}
Let $L$ be a separably tame valued field,
and let $K\subseteq L$ be a relatively algebraically closed subfield, equipped with the restriction of the valuation on $L$.
If the residue field extension $Lv|Kv$ is algebraic,
then $(K,v)$ is also a separably tame valued field,
and moreover, $vL/vK$ is torsion free and $Lv=Kv$.
\end{lemma}

The following lemma is preparation for the new step in the proof of Theorem~\ref{thm:EP}.
This method,
informally termed ``{\em wiggling}'', 
is applied 
in van de Schaaf's MSc thesis
\cite{vdS}
on separable taming,
and by Soto Moreno
\cite{SM25}
on relative quantifier elimination in separably algebraically maximal Kaplansky valued fields.
Also see the forthcoming paper by Jahnke and van der Schaaf 
\cite{JS25}.

\begin{lemma}\label{lem:wiggling}
Let $U$ be a nonempty open set in a topological field $L$ and let $K\subset L$ be a proper subfield.
Then $U\setminus K$ is not empty.
\end{lemma}
\begin{proof}
The subfield generated by any nontrivial open set $B$ in a field topology is the entire field,
since
$L=(U-U)\cdot((U-U)\setminus\{0\})^{-1}$.
\end{proof}

\begin{lemma}\label{lem:automatic_separability}
Let $K_{1},K_{2}$ be two separable field extensions of $K$,
and suppose that $K_{1}/K$ is separated.
Then every embedding $\iota:K_{1}\rightarrow K_{2}$ over $K$
is separable.
\end{lemma}
\begin{proof}
Let $c$ be a $p$-basis of $K$.
Since $K_{1}/K$ is separated,
by
Lemma~\ref{lem:Mac_Lane_IIa},
$c$ is a $p$-basis of $K_{1}$.
Since $K_{2}/K$ is separable,
by
Lemma~\ref{lem:Mac_Lane_II},
$c$ is $p$-independent in $K_{2}$.
Since $\iota$ is the identity on $K$,
$\iota(c)=c$,
which shows that $c$ is already a $p$-basis of the image of $\iota$.
\end{proof}

The following two lemmas provide cross-sections and sections (respectively)
in sufficiently saturated henselian valued fields.
Such maps give a additional structure to a valued field, and cross-sections especially have been part of standard approaches to Ax--Kochen/Ershov phenomena since the first papers.

\begin{lemma}[{\cite[Proposition 5.4]{vdD14}}]\label{lem:cross-section}
Let $s_{0}:\Delta\rightarrow K^{\times}$ be a partial cross-section of $v$
such that $\Delta$ is pure in $\Gamma_{v}$.
Then there is an elementary extension
$K\preceq K^{*}$
of valued fields,
with a cross-section
$s:vK^{*}\rightarrow K^{*\times}$
of the valuation on $K^{*}$
that extends $s_{0}$.
\end{lemma}

\begin{lemma}[{\cite[Proposition 4.5]{ADF23}}]\label{lem:section}
Let $K$ be a henselian valued field.
For every partial section
$\zeta_{0}:k_{0}\rightarrow K$ 
of
$\res_{v}$
with $Kv/k_{0}$ separable
there exists an elementary extension
$K\preceq K^{*}$
of valued fields,
with a section
$\zeta:K^{*}v^{*}\rightarrow K^{*}$
of
the residue map
$\res_{v^{*}}$
on $K^{*}$
that extends $\zeta_{0}$.
\end{lemma}

In the following, when we speak of embeddings we mean embeddings of valued fields.

\begin{theorem}[]\label{thm:EP}
$\Mod(\STVF^{\eq})$ has the \EP. 
\end{theorem}
{\bf Caveat}: the following is not intended to be a complete proof.
We follow very closely the embedding arguments in \cite{Kuh16,KuhlmannPal},
skipping lighting over the parts that are unchanged, 
but giving full details and making changes where necessary, around the new part of the argument.
Thus the reader might want to read this proof alongside those others.

\begin{proof}[Proof (modulo caveat)]
We work only in equal positive characteristic, since in equal characteristic zero
(and even in mixed characteristic)
every separably tame valued field is tame,
and the \EP\ becomes equivalent to the REP.
Suppose that we have
$K,K_{1},K_{2}\in\Mod(\STVF^{\eq})$,
where $K$ is a common valued field of $K_{1}$ and $K_{2}$,
satisfying the hypotheses~{\bf(i)--(v)} of the \EP.

By saturating the triple $(K,K_{1},K_{2})$,
if necessary,
we may assume by
Lemma~\ref{lem:cross-section}
that there is
a cross-section
$\chi:vK\rightarrow K^{\times}$,
and by
Lemma~\ref{lem:section}
that there is
a section
$\zeta:Kv\rightarrow K$.
By hypothesis~{\bf(iv)},
$v_{1}K_{1}/vK$ is torsion-free
and
$K_{1}v_{1}/Kv$ is separable.
Since
$\rho$ and $\sigma$ are embeddings,
also
$\rho(v_{1}K_{1})/vK$ is torsion-free
and
$\sigma(K_{1}v_{1})/Kv$ is separable.
Thus, 
by further saturating $K_{1}$ and $K_{2}$,
if necessary, 
again by
Lemmas~\ref{lem:cross-section}
and~\ref{lem:section},
there are extensions
$\chi_{1}:vK_{1}\rightarrow K_{1}^{\times}$
and $\chi_{2}:\rho(vK_{1})\rightarrow K_{2}^{\times}$
of $\chi$,
and extensions
$\zeta_{1}:K_{1}v\rightarrow K_{1}$
and $\zeta_{2}:\sigma(K_{1}v)\rightarrow K_{2}$
of $\zeta$.
Note that these saturation steps preserve the hypotheses of the \EP,
so these reductions are without loss of generality.

Let $K_{0}:=K(\zeta_{1}(K_{1}v),\chi_{1}(vK_{1}))^{\rac}$
be the relative algebraic closure in $K_{1}$ of the field generated over $K$
by the images of the section and cross-section.
As argued in \cite{KuhlmannPal},
$K_{0}/K$ is without transcendence defect,
and so
every finitely generated subextension $F/K$ of $K_{0}/K$
is strongly inertially generated,
by Theorem~\ref{thm:SIG}.
By compactness,
as in \cite{Kuh16},
there is an $\Lvlambda$-embedding
$\iota_{0}:K_{0}\rightarrow K_{2}$
such that
$\iota_{0}\circ\chi_{1}=\chi_{2}\circ\rho$
and
$\iota_{0}\circ\zeta_{1}=\zeta_{2}\circ\sigma$.
Thus $\iota_{0}$ induces $\rho$ and $\sigma$.

Note that $K_{0}/K$ is separated,
since $K_{1}v$ is perfect and $vK_{1}$ is $p$-divisible.
Moreover $K_{2}/K$ is separable by hypothesis~{\bf(i)},
so $\iota_{0}$ is automatically separable,
by Lemma~\ref{lem:automatic_separability},
i.e.~$K_{2}/\iota_{0}(K_{0})$ is separable.
By Lemma~\ref{lem:going_down2},
$K_{0}$ is also separably tame,
and $K_{1}/K_{0}$ is immediate.

Let $b=(b_{\mu})_{\mu<\nu}\subseteq K_{1}$
be a $p$-basis of $K_{1}$ over $K_{0}$.
For $\mu\leq\nu$,
let $K_{0,\mu}:=K_{0}(b_{\kappa})_{\kappa<\mu}^{\rac}$ be the relative algebraic closure of
$K_{0}(b_{\kappa})_{\kappa<\mu}$ in $K_{1}$.
Note that each $K_{0,\mu}$ is separably tame,
by Lemma~\ref{lem:going_down2}.
Since $K_{0,\nu}$ is the relative algebraic closure in $K_{1}$ of $K_{0}(b)$,
and $b$ is a $p$-basis of $K_{0,\nu}$ over $K_{0}$,
thus $K_{1}/K_{0,\nu}$ is separated,
using hypothesis~{\bf(i)}.
We will prove the following claim.
\begin{claim}\label{claim:1}
There is a separable $\Lvlambda$-embedding
$\iota_{0,\nu}:K_{0,\nu}\rightarrow K_{2}$
extending $\iota_{0}$.
\end{claim}
\begin{claimproof}
We will build a chain of separable $\Lvlambda$-embeddings
$\iota_{0,\mu}:K_{0,\mu}\rightarrow K_{2}$
for $\mu\leq\nu$.
We proceed inductively, noting that the base case is trivial since
$K_{0,0}=K_{0}$.
The limit stage is also easy:
a union of a chain of $\Lvlambda$-embeddings is an $\Lvlambda$-embedding.
We assume as an inductive hypothesis
that $\iota_{0}$ is already extended to a separable $\Lvlambda$-embedding
$\iota_{0,\mu}:K_{0,\mu}\rightarrow K_{2}$.
Let $c\in K_{0,\mu+1}$.
Then $c$ is separably algebraic over $K_{0,\mu}(b_{\mu})$.
By 
Theorem~\ref{thm:HR}
(Henselian Rationality),
there exists $d\in K_{0,\mu}(b_{\mu},c)$ such that $K_{0,\mu}(b_{\mu},c)^{h}=K_{0,\mu}(d)^{h}$.
It is clear that $d$ is inter-$p$-dependent with $b_{\mu}$ in $K_{1}$ over $K_{0,\mu}$.
In particular, $d$ is $p$-independent in $K_{1}$ over $K_{0,\mu}$.

Let $(d_{\delta})_{\delta<\alpha}$ be a pseudo-Cauchy sequence in $K_{0,\mu}$,
without pseudo-limit there,
of which $d$ is a pseudo-limit.
By Kuhlmann--Pal
(specifically by \cite[Lemma 3.11]{KuhlmannPal}),
and since $K_{0,\mu}$ is separably tame (so in particular separably algebraically maximal),
$(d_{\delta})_{\delta<\alpha}$ is of trancendental type.
By Kaplansky's second theorem,
\cite[Theorem 2]{Kaplansky},
its quantifier-free $\Lval$-type $q(x)$
over $K_{0,\mu}$
is implied by
formulas of the form $v(x-d_{\delta})\geq\gamma_{\delta}$,
where $\gamma_{\delta}=v(d_{\delta+1}-d_{\delta})$.
Any finitely many such formulas are already realised in $K_{0,\mu}$.
Let $q_{\iota}(x)$ be the image of $q(x)$ by translating the parameters in each formula by $\iota_{0,\mu}$.
Then $q_{\iota}(x)$ is implied by formulas of the form
$v(x-\iota_{0,\mu}(d_{\delta}))\geq\rho(\gamma_{\delta})$.
Any finitely many such formulas are realised in $\iota_{0,\mu}(K_{0,\mu})$,
and in particular $q_{\iota}(x)$ is consistent.
By saturation of $K_{2}$,
there is even a nontrivial ball $B$ in $K_{2}$ which is the set of realisations of $q_{\iota}(x)$.
Since $\impdeg(K_{1}/K)\leq\impdeg(K_{2}/K)$,
and by saturation,
i.e.~by hypotheses~{\bf(ii,iii)},
we have that
$K_{2}^{(p)}\iota_{0,\mu}(K_{0,\mu})$
is a proper subfield of $K_{2}$.
{Now comes the {\em wiggling}:}
there exists $d'\in B\setminus K_{2}^{(p)}\iota_{0,\mu}(K_{0,\mu})$,
by Lemma~\ref{lem:wiggling}.
Then $d'$ is $p$-independent in $K_{2}$ over $\iota_{0,\mu}(K_{0,\mu})$
and also realises $q_{\iota}(x)$.
Via the assignment $d\mapsto d'$ we extend $\iota_{0,\mu}$ to a separable $\Lvlambda$-embedding
$K_{0,\mu}(d)^{h}\rightarrow K_{2}$.

By the Primitive Element Theorem,
this already shows how to extend $\iota_{0,\mu}$ to a separable $\Lvlambda$-embedding into $K_{2}$
of any finite separably algebraic extension of $K_{0,\mu}(b_{\mu})$ inside $K_{0,\mu+1}$ .
By the Compactness Theorem, we extend $\iota_{0,\mu}$ to a separable embedding
$\iota_{0,\mu+1}:K_{0,\mu+1}\rightarrow K_{2}$,
as required for the inductive step.
By induction, there is a separable $\Lvlambda$-embedding
$\iota_{0,\nu}:K_{0,\nu}\rightarrow K_{2}$
extending $\iota_{0}$.
\end{claimproof}
The remaining extension of $\iota_{0,\nu}$ to
$\iota_{1}:K_{1}\rightarrow K_{2}$
is almost the same.
We construct an $\Lval$-embedding $\iota_{1}:K_{1}\rightarrow K_{2}$,
extending $\iota_{0,\nu}$,
by following the analogous arguments in \cite{Kuh16,KuhlmannPal},
that is by Henselian Rationality and Kaplansky's theory,
but without the ``{\em wiggling}'' argument.
Finally, we see that $\iota_{1}$ is automatically separable
since $K_{1}/K_{0,\nu}$ is separated,
by Lemma~\ref{lem:automatic_separability},
so $\iota_{1}$ is automatically an $\Lvlambda$-embedding.
\end{proof}

\begin{remark}
A similar Embedding Lemma
is applied in the case of separably algebraically maximal Kaplansky fields by Soto Moreno
(\cite{SM25})
to yield a relative quantifier elimination.
\end{remark}

\begin{remark}\label{rem:3.11}
During my discussions with Soto Moreno regarding his QE argument for separably algebraically maximal Kaplansky ("SAMK") valued fields,
we heard from F.~V.~Kuhlmann about a small error in the statement of \cite[Lemma 3.11]{KuhlmannPal}.
There is a hypothesis missing: the pseudo-Cauchy sequence should be supposed to not have a pseudo-limit in M.
This causes no significant problem for the above argument, which in any case is only a sketch.
Many thanks to Franz-Viktor Kuhlmann for this communication.
\end{remark}

\subsection{The resplendent model theory of separably tame valued fields}
\label{section:resplendent_STVF}

The Embedding Lemma yields 
model theoretic results,
specifically Ax--Kochen/Ershov principles and a transfer of decidability.
Moreover these results are resplendent
over the sorts $\mathbf{k}$ for the residue field, and $\bfGamma$ for the value group,
as we explain in this final subsection.

For any expansion $\Lang_{0}$ of $\Lval$, an {\em $(\mathbf{k},\bfGamma)$-expansion} of $\Lang_{0}$ is any expansion in which
\begin{itemize}
\item[---]
the residue field sort $\mathbf{k}$ is expanded to a language $\Lk\supseteq\Lring$,
and
\item[---]
the value group sort $\bfGamma$
is expanded to a language,
$\LGamma\supseteq\Loag$.
\end{itemize}
We emphasise that such a language is simply $\Lang_{0}$ expanded {\em by and only by} $\Lk$ on the residue field sort and $\LGamma$ on the value group sort.
Such an expansion will be denoted $\Lang_{0}(\Lk,\LGamma)$.
Usually $\Lang_{0}$ is either $\Lval$ or $\Lvlambda$.

Recall from \cite[\S2]{AF25_AE}
the notion of an {\em $\Lang$-fragment}%
\footnote{What are here called $\Lang$-fragments are also discussed in \cite{AF25_frag},
though in that paper they are just called ``fragments''.}%
, for a language $\Lang$:
a set $\Frag$ of $\Lang$-formulas that contains $\top$ and $\bot$,
that is closed under (finite) conjunctions and disjunctions,
and is closed under the substitution of one free variable for another.
For an $\Lang$-theory $T$ and an $\Lang$-fragment $\Frag$, we let $T_{\Frag}$ denote the intersection of the deductive closure $T^{\vdash}$ with $\Frag$.

A {\em fragment} is then a functor $\FFrag$ from a subcategory $\mathbb{L}$ of the category
of languages with inclusion
to the category
of sets with functions,
such that $\FFrag(\Lang)\subseteq\Form(\Lang)$,
for each $\Lang\in\mathbb{L}$.
If $T$ is an $\Lang$-theory where $\Lang\in\mathbb{L}$, we write $T_{\FFrag}=T_{\FFrag(\Lang)}$,
similarly if $M$ is an $\Lang$-structure we write $\Th_{\FFrag}(M)=\Th_{\FFrag(\Lang)}(M)=\Th(M)\cap\FFrag(\Lang)$.
For any language $\Lang$, for any fragment $\FFrag$,
and for any $\Lang$-structures $M_{1},M_{2}$ with common substructure $M$,
we write
$$M_{1}\Rrightarrow_{M}M_{2}\text{ in }\FFrag(\Lang)$$
to mean that $\FFrag$ is defined on both $\Lang$ and $\Lang(M)$, and moreover that
$\mathrm{Th}_{\FFrag}(M_{1,M})\subseteq\mathrm{Th}_{\FFrag}(M_{2,M})$,
where $M_{i,M}$ denotes the $\Lang(M)$ expansion of $M_{i}$ in which we interpret each new constant symbol by its corresponding element from $M$.

\begin{theorem}[{Main theorem for Separably Tame Fields}]\label{thm:STVF_main}
Let $\Lang=\Lvlambda(\Lk,\LGamma)$ be a $(\mathbf{k},\bfGamma)$-expansion of $\Lvlambda$.
Let $K_{1},K_{2}\in\Mod_{\Lang}(\STVF^{\eq})$ have common $\Lang$-substructure $K_{0}$
which as a valued field is
defectless,
and $v_{1}K_{1}/v_{0}K_{0}$ is torsion-free
and $K_{1}v_{1}/K_{0}v_{0}$ is relatively algebraically closed.
\begin{enumerate}[{\bf(I)}]
\item
$K_{1}\Rrightarrow_{K_{0}}K_{2}$
in $\Sent_{\exists}(\Lang)$
if and only if
\begin{enumerate}[{\bf(i)}]
\item
$k_{1}\Rrightarrow_{k_{0}}k_{2}$
in $\Sent_{\exists}(\Lang_{\mathbf{k}})$,
\item
$\Gamma_{1}\Rrightarrow_{\Gamma_{0}}\Gamma_{2}$
in $\Sent_{\exists}(\Lang_{\bfGamma})$,
and
\item
$\Impdeg(K_{1}/K_{0})\leq\Impdeg(K_{2}/K_{0})$.
\end{enumerate}
\item
$K_{1}\Rrightarrow_{K_{0}}K_{2}$
in $\Sent(\Lang)$
if and only if
\begin{enumerate}[{\bf(i)}]
\item
$k_{1}\Rrightarrow_{k_{0}}k_{2}$
in $\Sent(\Lang_{\mathbf{k}})$,
\item
$\Gamma_{1}\Rrightarrow_{\Gamma_{0}}\Gamma_{2}$
in $\Sent(\Lang_{\bfGamma})$,
and
\item
$\Impdeg(K_{1}/K_{0})=\Impdeg(K_{2}/K_{0})$.
\end{enumerate}
\end{enumerate}
\end{theorem}
\begin{proof}
For {\bf(I)},
the direction $\Rightarrow$ is almost trivial:
the interpretations of both $\mathbf{k}$ and $\bfGamma$ map
existential formulas to existential formulas.
Moreover,
if $\Impdeg(K_{0})=\infty$,
then certainly $\Impdeg(K_{1}/K_{0})=\Impdeg(K_{2}/K_{0})=\infty$.
Otherwise, suppose that $\Impdeg(K_{0})=m$ and let $c\in\pB[(K_{0})]$ be a $p$-basis of $K_{0}$.
If $\Impdeg(K_{1}/K_{0})\geq n$ then
$K_{1}\models\exists b=(b_{0},\ldots,b_{n-1})\;\lambda^{cb}_{\mathbf{0}}(1)=1$,
where $\mathbf{0}$ is the multi-index that is constantly zero.
By hypothesis, $K_{2}$ also models this sentence.
Therefore $\Impdeg(K_{2}/K_{0})\geq n$.
For the converse direction we suppose that {\bf(I:i,ii,iii)} hold.
Let $K_{1}^{*}\succeq K_{1}$ be an $\aleph_{0}$-saturated elementary extension,
and let $K_{2}^{*}\succeq K_{2}$ be an $|K_{1}^{*}|^{+}$-saturated elementary extension.
By saturation hypotheses,
there is
an $\Lk$-embedding $k^{*}_{1}\rightarrow k^{*}_{2}$ over $k_{0}$
and
an $\LGamma$-embedding $\Gamma^{*}_{1}\rightarrow\Gamma^{*}_{2}$ over $\Gamma_{0}$.
Then the three valued fields $K_{1}^{*},K_{2}^{*}$, with common valued subfield $K_{0}$,
satisfy the hypotheses of \EP.
The proof of {\bf(II)} is a standard back-and-forth argument,
making use of Theorem~\ref{thm:EP},
for example following the proof of \cite[Lemma 6.1]{Kuh16} or \cite[Lemma 4.1]{KuhlmannPal}.
\end{proof}

In Theorem~\ref{thm:sAKE} we will deduce that the class
$\Mod(\STVF^{\eq})$
satisfies the separable AKE principles $\sAKE^{\square}$,
for $\square\in\{\equiv,\equiv_{\exists},\preceq,\preceq_{\exists}\}$,
resplendently.
These principles are defined as follows.

\begin{definition}\label{def:sAKE}
Let $\Lang$ be
an expansion of a $(\mathbf{k},\bfGamma)$-expansion $\Lvlambda(\Lk,\LGamma)$ of $\Lvlambda$,
and let $\Lang_{0}\subseteq\Lang$.
Let $\mathbf{C}$ be a class of $\Lang$-structures
and let $\square\in\{\equiv,\equiv_{\exists},\preceq,\preceq_{\exists}\}$.
We say that $\mathbf{C}$
is an
{\em $\sAKE^{\square}$-class}
for the triple of languages $(\Lang_{0},\Lk,\LGamma)$,
if
for all $K_{1},K_{2}\in\mathbf{C}$
(where we additionally suppose $K_{1}\subseteq K_{2}$ in case $\square$ is either $\preceq$ or $\preceq_{\exists}$)
we have
\begin{itemize}
\item
$K_{1}\square K_{2}$
in $\Lang_{0}$
\end{itemize}
if and only if
\begin{itemize}
\item
$k_{1}\square k_{2}$
in $\Lk$,
\item
$\Gamma_{1}\square\Gamma_{2}$
in $\LGamma$,
and
\item
$\Impdeg(K_{1})=\Impdeg(K_{2})$.
\end{itemize}
In this case we say that $\mathbf{C}$
satisfies the {\em separable Ax--Kochen/Ershov principle}
$\sAKE^{\square}$
for the languages $\Lang_{0}$, $\Lk$, and $\LGamma$.
\end{definition}

\begin{theorem}[Resplendent Ax--Kochen/Ershov]\label{thm:sAKE}
Let $\Lang=\Lvlambda(\Lk,\LGamma)$ be a $(\mathbf{k},\bfGamma)$-expansion of $\Lvlambda$.
Let $\square\in\{\equiv,\equiv_{\exists},\preceq,\preceq_{\exists}\}$.
The class $\Mod_{\Lang}(\STVF^{\eq})$ of all $\Lang$-structures which expand separably tame valued fields of equal characteristic is
an
$\sAKE^{\square}$-class
for $(\Lang,\Lk,\LGamma)$.
\end{theorem}
\begin{proof}
Firstly, if $\square$ is $\equiv$,
the result follows from Theorem~\ref{thm:STVF_main}~{\bf(II)}
applied with $K_{0}=\mathbb{F}_{p}$ trivially valued.
Secondly, if $\square$ is $\equiv_{\exists}$,
the result follows from Theorem~\ref{thm:STVF_main}~{\bf(I)}
applied twice, with $K_{0}=\mathbb{F}_{p}$ trivially valued.
Thirdly, if $\square$ is $\preceq$,
the result follows from Theorem~\ref{thm:STVF_main}~{\bf(II)}
applied to $K_{0}=K_{1}$.
Finally, if $\square$ is $\preceq_{\exists}$,
the result follows from Theorem~\ref{thm:STVF_main}~{\bf(I)}
applied to $K_{0}=K_{1}$.
\end{proof}

\begin{proof}[Proof of Theorem~\ref{thm:intro_2}]
This is the special case of Theorem~\ref{thm:sAKE} in which
$\Lk=\Lring$,
$\LGamma=\Loag$,
and (thus)
$\Lang=\Lvlambda$.
\end{proof}

\begin{corollary}\label{cor:WEAK}
For $p\in\mathbb{P}$ and $\frakI\in\mathbb{N}\cup\{\infty\}$.
\begin{enumerate}[{\bf(i)}]
\item
$\STVF^{\eq}_{p,\frakI}$
is
resplendently complete
relative to the value group and residue field.
\item
$\STVF^{\eq}_{\lambda,p,\frakI}$
is
resplendently model complete 
relative to the value group and residue field.
\end{enumerate}
\end{corollary}
\begin{proof}
{\bf(i)} is just a particular case of the
$\sAKE^{\equiv}$ principle,
while ${\bf(ii)}$ is a particular case of the
$\sAKE^{\preceq}$ principle.
\end{proof}

Recall the sentences
$\chi_{p}$,
$\iota_{p,\leq\fraki}$,
and
$\iota_{p,\fraki}$,
and the theories
$\Fth_{p}$
and
$\Xth_{p,\frakI}$
introduced in Definition~\ref{def:XF}.
For $\Lang=\Lvlambda(\Lk,\LGamma)$,
let $\iota_{\mathbf{k}}$
denote the ``standard'' interpretation
$\Form(\Lk)\rightarrow\Form(\Lang)$
of the residue field
in a valued field,
the latter viewed as an $\Lval$-structure,
by relativising each $\Form(\Lk)$ to the sort $\mathbf{k}$.
Likewise let 
$\iota_{\bfGamma}$
denote the standard interpretation
$\Form(\LGamma)\rightarrow\Form(\Lang)$
for the value group on the sort $\bfGamma$.

\begin{definition}\label{def:IMP}
Let $\IMP$
(for ``imperfection'') be the
$\Lring$-fragment
consisting of all Boolean combinations of the $\Lring$-sentences
$\chi_{p}$ and $\iota_{p,\leq\fraki}$.
For $\Lang=\Lvlambda(\Lk,\LGamma)$ a $(\mathbf{k},\bfGamma)$-expansion of $\Lvlambda$,
let $\VFIMP(\Lang)$ be the $\Lang$-fragment generated by
$\iota_{\mathbf{k}}\mathrm{Sent}(\Lk)$
and
$\iota_{\bfGamma}\mathrm{Sent}(\LGamma)$,
and $\IMP$.
Then $\VFIMP$ is the fragment thus defined on the full subcategory of languages that are $(\mathbf{k},\bfGamma)$-expansions of $\Lvlambda$.
\end{definition}

\begin{theorem}\label{thm:monotone}
Let $\Lang=\Lvlambda(\Lk,\LGamma)$
be a $(\mathbf{k},\bfGamma)$-expansion of $\Lvlambda$,
and
let $K,L\in\Mod_{\Lang}(\STVF^{\eq})$.
If $\mathrm{Th}_{\VFIMP}(K)=\mathrm{Th}_{\VFIMP}(L)$
then $\mathrm{Th}(K)=\mathrm{Th}(L)$.
\end{theorem}
\begin{proof}
This is a reformulation of the
$\sAKE^{\equiv}$ principle
for $(\Lang,\Lk,\LGamma)$
from
Theorem~\ref{thm:sAKE}.
\end{proof}

The Hahn series fields
$k(\!(t^{\Gamma})\!)$, equipped with the $t$-adic valuation,
are natural examples of tame valued fields of equal characteristic,
with any given ``suitable'' pair of residue field $k$ and value group $\Gamma$.
By contrast, we lack such natural examples of separably tame valued fields with positive elementary imperfection degree $\frakI>0$.
Nevertheless, the following lemma justifies the existence of some example, for each suitable pair $k$ and $\Gamma)$.

\begin{lemma}\label{lem:surjective}
Let $(p,\frakI)\in\mathbb{P}\times(\mathbb{N}\cup\{\infty\})$,
let $k$ be any perfect field of characteristic $p$,
and
let $\Gamma$ be a $p$-divisible ordered abelian group.
There exists $K\models\STVF^{\eq}_{p,\frakI}$ with $Kv=k$ and $vK=\Gamma$.
\end{lemma}
\begin{proof}
We consider the immediate extension $k(\!(t^{\Gamma})\!)/k(t^{\Gamma})$,
both equipped with the $t$-adic valuation.
Let $B$ be a transcendence basis of this field extension, 
let $B_{0}\subseteq B$ be a subset of cardinality $\frakI$ if $\frakI<\infty$,
or of cardinality $\aleph_{0}$ if $\frakI=\infty$.
We notice that $B_{0}$ is a $p$-basis of $L_{0}:=k(t^{\Gamma},B_{0})$.
Let $L$ be a separable tamification of $L_{0}$
taken inside $k(\!(t^{\Gamma})\!)$,
i.e.~$L$ is a fixed field inside the separable closure of $L_{0}$ of a complement of the ramification group inside the absolute Galois group of $L_{0}$.
Then $L$ is separably tame,
with residue field $k$ and value group $G$,
and of elementary imperfection degree $\frakI$.
\end{proof}

\begin{theorem}\label{thm:elimination}
Let $\Lang=\Lvlambda(\Lk,\LGamma)$
be a $(\mathbf{k},\bfGamma)$-expansion of $\Lvlambda$.
There is an
``elimination''
function
$\epsilon:\mathrm{Sent}(\Lang)\rightarrow\VFIMP(\Lang)$
such that
$\STVF^{\eq}\models(\varphi\leftrightarrow\epsilon\varphi)$,
for all $\varphi\in\mathrm{Sent}(\Lang)$.
Moreover, if $\Lang$ is computable then $\epsilon$ may also be chosen to be computable.
\end{theorem}
\begin{proof}
We adopt the terminology of \cite{AF25_frag},
and consider the bridge
$$B=((\VFIMP(\Lang),\STVF^{\eq}),(\Form(\Lang),\STVF^{\eq}),\mathrm{id}).$$
First observe that the inclusion map
$\iota:\VFIMP(\Lang)\rightarrow\Form(\Lang)$
is an interpretation for $B$.
Moreover $B$
satisfies the monotonicity property
``$\mathrm{(mon)}$''
by Theorem~\ref{thm:monotone}.
By
\cite[Proposition 2.18]{AF25_frag},
therefore,
the required elimination $\epsilon$ exists.

Suppose now that $\Lang$ is computable,
then $\VFIMP(\Lang)$ and $\Form(\Lang)$ are computable $\Lang$-fragments,
and $\iota$ is computable.
Moreover $\STVF^{\eq}$ is computably enumerable (even computable),
in any case.
Thus, again by
\cite[Proposition 2.18]{AF25_frag},
the elimination $\epsilon$ may be chosen to be computable.
\end{proof}

\subsection{Computability-theoretic reductions}

Recall our Convention~\ref{convention:characteristic_exponent}
that $\hat{p}=p$ if $p\in\mathbb{P}$, and $\hat{p}=1$ if $p=0$.

Define
$\STVF^{\eq}(R,G,X):=(\STVF^{\eq}\cup\iota_{\mathbf{k}}R\cup\iota_{\bfGamma}G\cup X)_{\VFIMP}$.

\begin{theorem}[{Fixed characteristic, uniform in imperfection degree}]\label{thm:fixed_char_uniform_imp_decidability}
Let $\Lang=\Lvlambda(\Lk,\LGamma)$
be a $(\mathbf{k},\bfGamma)$-expansion of $\Lvlambda$,
let $p\in\mathbb{P}\cup\{0\}$,
let $R$ be an $\Lk$-theory of fields of characteristic $p$,
let $G$ be an $\LGamma$-theory of $\hat{p}$-divisible ordered abelian groups,
and
let $X$ be an $\IMP$-theory extending $\Xth_{p}$.
Suppose that $\Lang$ is computable.
Then 
\begin{enumerate}[{\bf(i)}]
\item
$\STVF^{\eq}(R,G,X)^{\vdash}\Teq R^{\vdash}\Toplus G^{\vdash}\Toplus(\Fth\cup X)_{\FIMP}$,
and
\item
$\STVF^{\eq}(R,G,X)^{\vdash}$ is decidable
if and only if
$R^{\vdash}$, $G^{\vdash}$, and $(\Fth\cup X)_{\FIMP}$ are decidable.
\end{enumerate}
\end{theorem}
\begin{proof}
We begin just like in the proof of Theorem~\ref{thm:elimination}.
Consider the bridge
$$B_{p}=((\VFIMP(\Lang),\STVF^{\eq}),(\Form(\Lang),\STVF^{\eq}),\mathrm{id}).$$
Observe that
$\Lang$ is computable,
so $\VFIMP(\Lang)$ and $\Form(\Lang)$ are computable,
$\iota$ is computable,
and
$\STVF^{\eq}$ is computably enumerable (even computable).
The bridge $B_{p}$ satisfies ``surjectivity'' by Lemma~\ref{lem:surjective},
and $B_{p}$ satisfies ``monotonicity'' by Theorem~\ref{thm:monotone}.
We have verified the hypotheses of \cite[Corollary 2.23]{AF25_frag}
for the arch $A=(B_{p},B_{p},\iota)$.
Applying that result,
we obtain
\begin{enumerate}[{\bf(I)}]
\item
	$B_{p}$ admits a computable elimination (this already follows from Theorem~\ref{thm:elimination}).
\item
	$\STVF^{\eq}(R,G,X)^{\vdash}\meq\STVF^{\eq}(R,G,X)_{\VFIMP}$
\end{enumerate}
It's also rather clear that
$\STVF^{\eq}(R,G,X)_{\VFIMP}\meq(R\sqcup G\sqcup X)_{\VFIMP}$,
and weaking our sense of equivalence to that of Turing equivalence we have
$(R\sqcup G\sqcup X)_{\VFIMP}\Teq R^{\vdash}\Toplus G^{\vdash}\Toplus(\Fth\cup X)_{\FIMP}$.
Combining these with {\bf(II)},
we obtain {\bf(i)}.
Finally,
{\bf(ii)} follows immediately from {\bf(i)}.
\end{proof}

\begin{corollary}[{Fixed characteristic and fixed/arbitrary imperfection degree}]\label{cor:fixed_char_and_fixed_imp_decidability}
Let $\Lang=\Lvlambda(\Lk,\LGamma)$
be a $(\mathbf{k},\bfGamma)$-expansion of $\Lvlambda$,
let $(p,\frakI)\in\{(0,0)\}\cup(\mathbb{P}\times(\mathbb{N}\cup\{\infty\}))$,
let $R$ be an $\Lk$-theory of fields of characteristic $p$,
and
let $G$ be an $\LGamma$-theory of $\hat{p}$-divisible ordered abelian groups.
Suppose that $\Lang$ is computable.
Then 
\begin{enumerate}[{\bf(I)}]
\item
	\begin{enumerate}[{\bf(i)}]
	\item
	$\STVF^{\eq}(R,G,\Xth_{p,\frakI})^{\vdash}\Teq R^{\vdash}\Toplus G^{\vdash}$,
	and
	\item
	$\STVF^{\eq}(R,G,\Xth_{p,\frakI})^{\vdash}$ is decidable
	if and only if
	$R^{\vdash}$ and $G^{\vdash}$ are decidable.
	\end{enumerate}
\item
	\begin{enumerate}[{\bf(i)}]
	\item
	$\STVF^{\eq}(R,G,\Xth_{p})^{\vdash}\Teq R^{\vdash}\Toplus G^{\vdash}$,
	and
	\item
	$\STVF^{\eq}(R,G,\Xth_{p})^{\vdash}$ is decidable
	if and only if
	$R^{\vdash}$ and $G^{\vdash}$ are decidable.
	\end{enumerate}
\end{enumerate}
\end{corollary}
\begin{proof}
Both $(\Fth_{p,\frakI})_{\FIMP}$ and $(\Fth_{p})_{\FIMP}$ are decidable.
\end{proof}

This immediately implies Corollary~\ref{cor:intro_a}.

\begin{question}
How may we adapt Theorem~\ref{thm:fixed_char_uniform_imp_decidability}
so that it is uniform in the characteristic $p$?
\end{question}

\section*{Acknowledgements}

The material of sections 2 and 3 in this paper was drawn from the author's doctoral thesis,
completed in 2013 at the University of Oxford.
The author would like to thank her doctoral supervisor Jochen Koenigsmann for his support and guidance,
and
Arno Fehm,
Franziska Jahnke,
Franz--Viktor Kuhlmann,
and
Silvain Rideau-Kikuchi
for comments on an earlier version.
Thanks are due to
Paulo Andrés Soto Moreno
and
Jonas van der Schaaf,
both of whom have been applying these kinds of ``{\em wiggling}'' arguments.
Discussions with all of them have greatly improved section~\ref{section:STVF} --- many thanks.

The manuscript was completed while the author was a Young Research Fellow at the Mathematics Münster Cluster of Excellence,
at the University of Münster.
She would like to extend warm thanks to the mathematical community of Münster,
especially to the Cluster of Excellence and her host Franziska Jahnke:
Funded by the Deutsche Forschungsgemeinschaft
(DFG, German Research Foundation)
under Germany's Excellence Strategy
EXC 2044-390685587, Mathematics Münster: Dynamics - Geometry - Structure.
This sabbatical was made possible by the award of a
{Congés pour Recherches ou Conversions Thématiques}
by the
{Conseil National des Universités},
to whom her thanks go also.

\def\bibfont{\footnotesize}
\bibliographystyle{plain}

\end{document}